\pdfoutput=1
\documentclass[11pt]{article}
\usepackage{authblk}
\usepackage[toc,page]{appendix}
\usepackage[top=2cm, bottom=2cm, left=2cm, right=2cm]{geometry}

\usepackage{color}
\usepackage{helvet}         
\usepackage{courier}        
\usepackage{type1cm}        

\usepackage{framed}
\usepackage{tikz}
\usepackage{makeidx}         
\usepackage{graphicx}        
\usepackage{multicol}        
\usepackage[bottom]{footmisc}

\usepackage{amsmath}
\usepackage{amssymb}
\usepackage{bbold}
\usepackage{amsthm}
\usepackage{subcaption}
\usepackage{sidecap}
\usepackage{floatrow}
\usepackage{pdflscape}
\usepackage{pdflscape}
\usepackage{comment}
\usepackage[font=small]{caption}
\usepackage{enumitem}
\usepackage{esint}

\usepackage{scalerel}

\newtheorem{theorem}{Theorem}
\newtheorem{corollary}[theorem]{Corollary}
\newtheorem{lemma}[theorem]{Lemma}
\newtheorem{conjecture}[theorem]{Conjecture}
\newtheorem{proposition}[theorem]{Proposition}
\newtheorem{proposition*}{Proposition}
\newtheorem{lemma*}{Lemma}

\theoremstyle{remark}

\newtheorem{remark}[theorem]{\bf Remark}
\newtheorem{definition}[theorem]{\bf Definition}

\numberwithin{theorem}{section}
\numberwithin{question}{section}
\numberwithin{figure}{section}
\numberwithin{equation}{section}

\begin{document}

\title{Radial BPZ equations and partition functions of FK-Ising interfaces conditional on one-arm event}
\bigskip{}
\author[1]{Yu Feng\thanks{yufeng\_proba@163.com}}
\author[1]{Hao Wu\thanks{hao.wu.proba@gmail.com. Supported by Beijing Natural Science Foundation (JQ20001).}}
\affil[1]{Tsinghua University, China}
\date{}

%

%

\maketitle

\begin{center}
\begin{minipage}{0.95\textwidth}
\abstract{
Radial BPZ equations come naturally when one solves Dub\'{e}dat's commutation relation in the radial setting. We construct positive solutions to radial BPZ equations and show that partition functions of FK-Ising interfaces in a polygon conditional on a one-arm event are positive solutions to radial BPZ equations. 
 }

\bigskip{}


\end{minipage}
\end{center}

\global\long\def\w{\mathrm{w}}
\global\long\def\d{\mathrm{d}}
\global\long\def\f{\mathrm{f}}
\global\long\def\t{\theta}
\global\long\def\vt{\vartheta}
\global\long\def\CR{\mathrm{CR}}
\global\long\def\ST{\mathrm{ST}}
\global\long\def\SF{\mathrm{SF}}
\global\long\def\cov{\mathrm{cov}}
\global\long\def\dist{\mathrm{dist}}
\global\long\def\SLE{\mathrm{SLE}}
\global\long\def\hSLE{\mathrm{hSLE}}
\global\long\def\CLE{\mathrm{CLE}}
\global\long\def\GFF{\mathrm{GFF}}
\global\long\def\inte{\mathrm{int}}
\global\long\def\ext{\mathrm{ext}}
\global\long\def\inrad{\mathrm{inrad}}
\global\long\def\outrad{\mathrm{outrad}}
\global\long\def\dimH{\mathrm{dim}}
\global\long\def\capa{\mathrm{cap}}
\global\long\def\diam{\mathrm{diam}}
\global\long\def\free{\mathrm{free}}
\global\long\def\hF{{}_2\mathrm{F}_1}
\global\long\def\ghF{{}_3\mathrm{F}_2}
\global\long\def\simple{\mathrm{simple}}
\global\long\def\even{\mathrm{even}}
\global\long\def\odd{\mathrm{odd}}
\global\long\def\st{\mathrm{ST}}
\global\long\def\LZalphawired{\LZ^{(\mathfrak{r})}}
\global\long\def\LQalphawired{\LQ^{(\mathfrak{r},\mathrm{w})}}
\global\long\def\usf{\mathrm{USF}}
\global\long\def\Leb{\mathrm{Leb}}
\global\long\def\LP{\mathrm{LP}}
\global\long\def\coulomb{\LH}
\global\long\def\coulombnew{\LG}
\global\long\def\kfunc{p}
\global\long\def\OO{\mathcal{O}}
\global\long\def\Dist{\mathrm{Dist}}
\global\long\def\ball{D}
\global\long\def\cst{\mathrm{C}_{\kappa}^{(\mathfrak{r})}}

\global\long\def\eps{\epsilon}
\global\long\def\ov{\overline}
\global\long\def\U{\mathbb{U}}
\global\long\def\T{\mathbb{T}}
\global\long\def\HH{\mathbb{H}}
\global\long\def\LA{\mathcal{A}}
\global\long\def\LB{\mathcal{B}}
\global\long\def\LC{\mathcal{C}}
\global\long\def\LD{\mathcal{D}}
\global\long\def\LF{\mathcal{F}}
\global\long\def\LK{\mathcal{K}}
\global\long\def\LE{\mathcal{E}}
\global\long\def\LG{\mathcal{G}}
\global\long\def\LI{\mathcal{I}}
\global\long\def\LJ{\mathcal{J}}
\global\long\def\LL{\mathcal{L}}
\global\long\def\LM{\mathcal{M}}
\global\long\def\LN{\mathcal{N}}
\global\long\def\LQ{\mathcal{Q}}
\global\long\def\LR{\mathcal{R}}
\global\long\def\LT{\mathcal{T}}
\global\long\def\LS{\mathcal{S}}
\global\long\def\LU{\mathcal{U}}
\global\long\def\LV{\mathcal{V}}
\global\long\def\LW{\mathcal{W}}
\global\long\def\LX{\mathcal{X}}
\global\long\def\LY{\mathcal{Y}}
\global\long\def\LZ{\mathcal{Z}}
\global\long\def\LZtwo{\mathcal{Z}_{\includegraphics[scale=0.2]{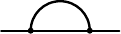}}}
\newcommand{\Etwo}{\mathbb{E}_{\includegraphics[scale=0.2]{figures/link-0}}}
\newcommand{\QQtwo}{\QQ_{\includegraphics[scale=0.2]{figures/link-0}}}
\newcommand{\LZtwor}{\LZtwo^{(\mathfrak{r})}}

\global\long\def\LZfoura{\mathcal{Z}_{\includegraphics[scale=0.2]{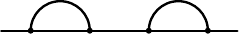}}}
\global\long\def\LZfouraone{\mathcal{Z}_{\includegraphics[scale=0.2]{figures/link-1};1}}
\global\long\def\LZfouratwo{\mathcal{Z}_{\includegraphics[scale=0.2]{figures/link-1};2}}
\global\long\def\LZfourathree{\mathcal{Z}_{\includegraphics[scale=0.2]{figures/link-1};3}}

\global\long\def\LZfourb{\mathcal{Z}_{\includegraphics[scale=0.2]{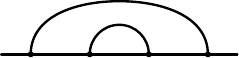}}}
\global\long\def\LZfourbfusionl{\mathcal{Z}_{\includegraphics[scale=0.6]{figures/link211}}}
\global\long\def\LZfourbfusionr{\mathcal{Z}_{\includegraphics[scale=0.6]{figures/link112}}}
\global\long\def\LEfourb{\mathbb{E}_{\includegraphics[scale=0.2]{figures/link-2}}}
\global\long\def\LPfourb{\mathbb{P}_{\includegraphics[scale=0.2]{figures/link-2}}}

\global\long\def\LEfourbfusionl{\mathbb{E}_{\includegraphics[scale=0.6]{figures/link211}}}
\global\long\def\LEfourbfusionr{\mathbb{E}_{\includegraphics[scale=0.6]{figures/link112}}}
\global\long\def\LEfourbfusion{\mathbb{E}_{\includegraphics[scale=0.5]{figures/link22}}}

\global\long\def\LZfourbone{\mathcal{Z}_{\includegraphics[scale=0.2]{figures/link-2};1}}
\global\long\def\LZfourbtwo{\mathcal{Z}_{\includegraphics[scale=0.2]{figures/link-2};2}}
\global\long\def\LZfourbthree{\mathcal{Z}_{\includegraphics[scale=0.2]{figures/link-2};3}}

\global\long\def\LFfoura{\mathcal{F}_{\includegraphics[scale=0.2]{figures/link-1}}}
\global\long\def\LFfourb{\mathcal{F}_{\includegraphics[scale=0.2]{figures/link-2}}}

\global\long\def\LZalphar{\mathcal{Z}_{\alpha}^{(\mathfrak{r})}}
\global\long\def\QQalphar{\mathbb{Q}_{\alpha}^{(\mathfrak{r})}}
\global\long\def\LZalpharwired{\mathcal{Z}_{\alpha; \mathrm{w}}^{(\mathfrak{r})}}
\global\long\def\QQalpharwired{\mathbb{Q}_{\alpha; \mathrm{w}}^{(\mathfrak{r})}}
\global\long\def\LH{\mathcal{H}}
\global\long\def\LJ{\mathcal{J}}
\global\long\def\R{\mathbb{R}}
\global\long\def\C{\mathbb{C}}
\global\long\def\N{\mathbb{N}}
\global\long\def\Z{\mathbb{Z}}
\global\long\def\E{\mathbb{E}}
\global\long\def\PP{\mathbb{P}}
\global\long\def\QQ{\mathbb{Q}}
\global\long\def\A{\mathbb{A}}
\global\long\def\one{\mathbb{1}}
\global\long\def\bn{\mathbf{n}}
\global\long\def\MR{MR}
\global\long\def\cond{\,|\,}
\global\long\def\la{\langle}
\global\long\def\ra{\rangle}
\global\long\def\tree{\Upsilon}
\global\long\def\prob{\mathbb{P}}
\global\long\def\hm{\mathrm{Hm}}
\global\long\def\cross{\mathrm{Cross}}

\global\long\def\sf{\mathrm{SF}}
\global\long\def\wr{\varrho}

\global\long\def\Im{\operatorname{Im}}
\global\long\def\Re{\operatorname{Re}}

\global\long\def\ud{\mathrm{d}}
\global\long\def\pder#1{\frac{\partial}{\partial#1}}
\global\long\def\pdder#1{\frac{\partial^{2}}{\partial#1^{2}}}
\global\long\def\der#1{\frac{\ud}{\ud#1}}

\global\long\def\bZnn{\mathbb{Z}_{\geq 0}}

\global\long\def\Vfunc{\LG}
\global\long\def\gfunc{g^{(\rr)}}
\global\long\def\hfunc{h^{(\rr)}}

\global\long\def\SimplexInt{\rho}
\global\long\def\CubeInt{\widetilde{\rho}}

\global\long\def\ii{\mathfrak{i}}
\global\long\def\ee{\mathrm{e}}
\global\long\def\rr{\mathfrak{r}}
\global\long\def\chamber{\mathfrak{X}}
\global\long\def\Wchamber{\mathfrak{W}}

\global\long\def\SimplexIntKappa8{\SimplexInt}

\global\long\def\nested{\boldsymbol{\underline{\Cap}}}
\global\long\def\unnested{\boldsymbol{\underline{\cap\cap}}}
\global\long\def\unnested{\boldsymbol{\underline{\cap\cap}}}

\global\long\def\acycle{\vartheta}
\global\long\def\bcycle{\tilde{\acycle}}
\global\long\def\Gloop{\Theta}

\global\long\def\metric{\mathrm{dist}}

\global\long\def\adj#1{\mathrm{adj}(#1)}

\global\long\def\bs{\boldsymbol}

\global\long\def\edge#1#2{\langle #1,#2 \rangle}
\global\long\def\graph{G}

\newcommand{\conn}{\vartheta}
\newcommand{\hatconn}{\widehat{\vartheta}_{\mathrm{RCM}}}
\newcommand{\realpt}{\smash{\mathring{x}}}
\newcommand{\corrind}{\LC}
\newcommand{\bssymb}{\pi}
\newcommand{\PRCM}{\mu}
\newcommand{\coeff}{p}
\newcommand{\MainConst}{C}

\global\long\def\removeLink{/}

\section{Introduction}

To describe the scaling limit of random interfaces in planar critical lattice models, Schramm realized that it is equivalent to classifying random planar curves with conformal invariance and domain Markov property. In a simply connected domain with two marked boundary points, these properties impose that the chordal Loewner driving function of such a random curve has to be a multiple of Brownian motion, and this gives the definition of chordal SLE~\cite{SchrammFirstSLE}. 

After classifying the random curves with conformal invariance and domain Markov property in a simply connected domain with two marked points on the boundary, it is natural to try to classify random curves with these properties in a simply connected domain with more marked points which correspond to the scaling limit of 
random interfaces in a polygon. We say that $(\Omega; x_1, \ldots, x_p)$ is a polygon if $\Omega\subsetneq\C$ is simply connected and $x_1, \ldots, x_p\in\partial\Omega$ are distinct points lying counterclockwise along the boundary. We assume that $\partial\Omega$ is locally connected. We say $(\Omega; x_1, \ldots, x_{p})$ is a nice polygon if the marked boundary points $x_1, \ldots, x_p$ lie on sufficiently regular boundary segments, e.g. $C^{1+\eps}$ for some $\eps>0$. The most often used polygon is the upper half-plane $(\HH; x_1, \ldots, x_p)$ with $x_1<\cdots<x_p$. Dub\'{e}dat analyzed random curves in polygons with conformal invariance, domain Markov property, and a technical requirement ``absolute continuity" in~\cite{DubedatCommutationSLE}. These properties give a commutation relation on the infinitesimal generators of the curves. In particular, such commutation relation results in a system of chordal Belavin-Polyakov-Zamolodchikov (BPZ) equations: for all $1\le j\le p$, 
	\begin{align}\label{eqn::chordalBPZ}
		\frac{\kappa}{2} \frac{\partial_j^2\LZ}{\LZ}
		+ \sum_{\ell\neq j} \left( \frac{2}{x_{\ell}-x_{j}}\frac{\partial_{\ell}\LZ}{\LZ}
		- \frac{2h}{(x_{\ell}-x_{j})^{2}} \right) =  0.
	\end{align}
To classify random curves in a polygon with conformal invariance and domain Markov property, one needs to understand positive solutions to chordal BPZ equations~\eqref{eqn::chordalBPZ}. Since~\cite{DubedatCommutationSLE}, there has been active research on the classification of solutions to chordal BPZ equations and on their relation to planar critical lattice models~\cite{GrahamSLE,LawlerPartitionFunctionsSLE,FloresKlebanPDE2,FloresKlebanPDE3,
KytolaPeltolaPurePartitionFunctions,FloresSimmonsKlebanZiffCrossingProbaPolygon,
IzyurovIsingMultiplyConnectedDomains,
PeltolaWuGlobalMultipleSLEs,WuHyperSLE,FSKconnectivityweights,IzyurovMultipleFKIsing,
AngHoldenSunYu2023,PeltolaWuCrossingProbaIsing,SunYuMutipleSLE2023,
LiuPeltolaWuUST,FengPeltolaWuConnectionProbaFKIsing,FengLiuPeltolaWu2024}.

In contrast to the chordal setting, commutation relation in the radial setting is less explored.  We say $(\Omega; x_1, \ldots, x_p; z)$ is a polygon with an interior point if $(\Omega; x_1, \ldots, x_p)$ is a polygon and $z\in\Omega$. 
The most often used polygon with an interior point is the unit disc $(\U; \exp(\ii\theta_1), \ldots, \exp(\ii\theta_p); 0)$ with 
$\bs{\theta}=(\theta_1, \ldots, \theta_p)\in \LX_p$ where 
\[\LX_p=\{\bs{\theta}=(\theta_1, \ldots, \theta_p)\in\R^p: \theta_1<\theta_2<\cdots<\theta_p<\theta_1+2\pi\}.\]
In~\cite{DubedatCommutationSLE}, Dub\'{e}dat also derived commutation relation in such radial setting. In particular, the commutation relation in the radial setting also gives a system of radial BPZ equations~\cite{DubedatCommutationSLE, WangWuCommutationRelation, ZhangMultipleRadialSLE0}: there exists a constant $\aleph\in\R$, for all $1\le j\le p$, 
\begin{align}\label{eqn::radialBPZ}	
	\frac{\kappa}{2}\frac{\partial_{j}^2\LZ}{\LZ}+\sum_{\ell\neq j}\left(\cot((\theta_{\ell}-\theta_j)/2)\frac{\partial_{\ell}\LZ}{\LZ}-\frac{(6-\kappa)/(4\kappa)}{\left(\sin((\theta_{\ell}-\theta_j)/2)\right)^2}\right)=\aleph. 
	\end{align}
Different from the chordal setting, the system of radial BPZ equations~\eqref{eqn::radialBPZ} has one more degree of freedom on the choice of the constant $\aleph$ on the right-hand side. It is an interesting question to classify solutions to radial BPZ equations~\eqref{eqn::radialBPZ} and to understand the role of the constant $\aleph$. 
The authors in~\cite{HealeyLawlerNSidedRadialSLE} analyzed multi-sided radial $\SLE_{\kappa}$ whose partition function is a solution to the system of radial BPZ equations~\eqref{eqn::radialBPZ} with $\aleph=\frac{1-N^2}{2\kappa}$. 
The authors in~\cite{WangWuCommutationRelation} studied commutation relation and found all solutions to the system of radial BPZ equations~\eqref{eqn::radialBPZ} with $N=1$ and $\aleph\in\R$. 
The author in~\cite{ZhangMultipleRadialSLE0} analyzed solutions to the semi-classical limit ($\kappa=0$) of radial BPZ equations~\eqref{eqn::radialBPZ}. See~\cite{CardyCalogeroSutherlandModel,DoyonCardyCalogeroSutherland,FSKZ11,FKZ12} for other results about radial BPZ equations.

As we mentioned above, commutation relation comes naturally when one investigates the scaling limit of planar critical lattice models. Thus, solutions to radial BPZ equations should have a connection to planar critical lattice models. This is the focus of this article. 
It turns out that the scaling limit of interfaces of planar critical random-cluster model conditional on one-arm event corresponds to a partition function which is a positive solution to radial BPZ equations~\eqref{eqn::radialBPZ} with a specific constant $\aleph$ that we describe below.  

\subsection{Random-cluster model}
We fix a polygon with an interior point $(\Omega; x_1, \ldots, x_{2N}; z)$. Suppose $(\Omega^{\delta}; x_1^{\delta}, \ldots, x_{2N}^{\delta}; z^{\delta})$ is a sequence of discrete domains on $\delta\Z^2$ converges to $(\Omega; x_1, \ldots, x_{2N}; z)$ in the close-Carath\'{e}odory sense (see Definition~\ref{def:closeCara}).  
We denote by 
\[\PP^{\delta}_{\Omega}=\PP^{\delta}_{(\Omega; x_1, \ldots, x_{2N})}\]
the law of critical random-cluster model in $(\Omega^{\delta}; x_1^{\delta}, \ldots, x_{2N}^{\delta})$ with cluster-weight $q\in [1,4)$ and alternating boundary condition: 
\begin{equation}\label{eqn::bc}
(x_{2j-1}^{\delta}x_{2j}^{\delta})\text{ is wired},\qquad \text{for all }j\in\{1,\ldots, N\},
\end{equation}
and these $N$ wired arcs are not wired outside of $\Omega^{\delta}$ (see details in Section~\ref{subsec::pre_rcm}). 
Under such boundary condition, there exist $N$ interfaces $(\eta^{(1,\delta)},\ldots,\eta^{(N,\delta)})$ on the medial graph $\Omega^{\delta,\diamond}$ connecting the marked points $\{x_1^{\delta, \diamond}, \ldots, x_{2N}^{\delta, \diamond}\}$ pairwise. The goal of this article is to derive the scaling limit of these interfaces conditional on the following one-arm event: 
\begin{equation}\label{eqn::onearm_event_def}
\LA^{\delta}=\LA^{\delta}(\Omega^{\delta}; x_1^{\delta}, \ldots, x_{2N}^{\delta}; z^{\delta})=\{\exists \text{ an open path connecting }z^{\delta}\text{ to }\cup_{j=1}^N(x_{2j-1}^{\delta}x_{2j}^{\delta})\}.
\end{equation}

\begin{conjecture} \label{conj::cvg_interfaces}
Fix a polygon with an interior point $(\Omega; x_1, \ldots, x_{2N}; z)$ and suppose a sequence of medial domains $(\Omega^{\delta, \diamond}; x_1^{\delta, \diamond}, \ldots, x_{2N}^{\delta, \diamond}; z^{\delta, \diamond})$ converges to $(\Omega; x_1, \ldots, x_{2N}; z)$ in the close-Carath\'{e}odory sense. Consider critical random-cluster model on the primal domain $(\Omega^{\delta}; x_1^{\delta}, \ldots, x_{2N}^{\delta})$ with alternating boundary condition~\eqref{eqn::bc}. 
The cluster-weight $q$ and parameter $\kappa$ are related through
\begin{equation}\label{eqn::qkappa}
q=4\cos^2(4\pi/\kappa)\in [1,4), \quad \kappa\in (4,6]. 
\end{equation}
Let $\eta_j^{\delta}$ be the interface starting from $x_j^{\delta,\diamond}$. Let $\varphi$ be any conformal map from $\Omega$ onto $\U$ such that $\varphi(z)=0$ and denote $\exp(\ii\theta_j)=\varphi(x_j)$ for $j\in\{1, \ldots, 2N\}$ such that $\theta_1<\cdots<\theta_{2N}<\theta_1+2\pi$. Then the law of $\eta_j^{\delta}$ conditional on the one-arm event $\LA^{\delta}$ in~\eqref{eqn::onearm_event_def} converges weakly under the topology induced by $\dist$ in~\eqref{eq::curve_metric} to the image under $\varphi^{-1}$ of the radial Loewner chain with the following driving function, up to the first time either $\varphi(x_{j-1})$ or $\varphi(x_{j+1})$ is disconnected from the origin:
\begin{align} \label{eqn::driving_function_FK_conditional}
\begin{cases}
\ud \xi_t=\sqrt{\kappa}\ud B_t+\kappa(\partial_j\log \LG^{(\mathfrak{r}_1)})(V_t^{(1)}, \ldots, V_t^{(j-1)}, \xi_t, V_t^{(j+1)}, \ldots, V_t^{(2N)})\ud t,\quad \xi_0=\theta_j;\\
\ud V_t^{(i)}=\cot((V_t^{(i)}-\xi_t)/2)\ud t, \quad V_0^{(i)}=\theta_i, \quad\text{for }i\in\{1, \ldots, j-1, j+1, \ldots, 2N\},
\end{cases}
\end{align}
where $\mathfrak{r}_1=\mathfrak{r}_1(\kappa)$ is the one-arm exponent for conformal loop ensemble~\cite{SchrammSheffieldWilsonConformalRadii}:
\[\mathfrak{r}_1(\kappa):=\frac{(3\kappa-8)(8-\kappa)}{32\kappa},\]
the partition function $\LG^{(\mathfrak{r}_1)}$ is defined in Definition~\ref{def::partition_random_cluster} and is a solution to the system of radial BPZ equations~\eqref{eqn::radialBPZ} with $\aleph=\frac{16-\kappa^2}{32\kappa}$: for all $1\le j\le 2N$, 
\begin{align}\label{eqn::radialBPZ_onearm}	
	\frac{\kappa}{2}\frac{\partial_{j}^2\LG^{(\mathfrak{r}_1)}}{\LG^{(\mathfrak{r}_1)}}+\sum_{\ell\neq j}\left(\cot((\theta_{\ell}-\theta_j)/2)\frac{\partial_{\ell}\LG^{(\mathfrak{r}_1)}}{\LG^{(\mathfrak{r}_1)}}-\frac{(6-\kappa)/(4\kappa)}{\left(\sin((\theta_{\ell}-\theta_j)/2)\right)^2}\right)=\frac{16-\kappa^2}{32\kappa}. 
	\end{align}
\end{conjecture}

The partition function $\LG^{(\mathfrak{r}_1)}$ in Conjecture~\ref{conj::cvg_interfaces} is not explicit in general, but it has a simple explicit expression when $N=1$: up to a multiplicative constant, 
\begin{equation}
\LG^{(\mathfrak{r}_1)}(\theta_1, \theta_2)=\left(\sin\left(\left(\theta_2-\theta_1\right)/2\right)\right)^{1-6/\kappa}\left(\sin\left(\left(\theta_2-\theta_1\right)/4\right)\right)^{8/\kappa-1}.
\end{equation}

\begin{theorem}\label{thm::FKIsing_cvg}
Conjecture~\ref{conj::cvg_interfaces} holds for FK-Ising model $q=2$ with $\kappa=16/3$. 
\end{theorem}

The proof of Theorem~\ref{thm::FKIsing_cvg} relies on three inputs: 
1st. a general construction of positive solutions to radial BPZ equations in Proposition~\ref{prop::solutions_evenr2}; 
2nd. the asymptotic analysis of probabilities of one-arm events of the FK-Ising model in Lemma~\ref{lem::limit_one_point} and 
3rd. the convergence of FK-Ising interfaces without conditioning~\cite{BeffaraPeltolaWuUniqueness, FengPeltolaWuConnectionProbaFKIsing}. 
We will explain the construction of positive solutions in Section~\ref{subsec::intro_solutions} and then explain the strategy of the proof of Theorem~\ref{thm::FKIsing_cvg} in Section~\ref{subsec::intro_strategy}. 

\subsection{Positive solutions for radial BPZ equations}
\label{subsec::intro_solutions}
In this section, we construct positive solutions to radial BPZ equations~\eqref{eqn::radialBPZ} using global multiple $\SLE_{\kappa}$. To this end, we first introduce global multiple SLEs.

For $\kappa\in (0,8)$, a chordal $\SLE_{\kappa}$ 
is a random continuous non-self-crossing curve in a simply connected domain $\Omega$ connecting two prime ends $x_1, x_2$ on the boundary $\partial\Omega$. We call such process chordal $\SLE_{\kappa}$ in $(\Omega; x_1, x_2)$ for short.
Suppose $(\Omega; x_1, \ldots, x_{2N})$ is a polygon.
We consider curves $(\eta^{(1)}, \ldots, \eta^{(N)})$ in $\overline{\Omega}$ each of which connects two distinct points among $\{x_1, x_2, \ldots, x_{2N}\}$ in such a way that they do not cross each other. The curves $(\eta^{(1)}, \ldots, \eta^{(N)})$ can have various planar connectivity patterns, which we describe in terms of planar link patterns
$\alpha=\{\{a_1, b_1\}, \ldots, \{a_N, b_N\}\}$ where $\{a_1, b_1, \ldots, a_N, b_N\}=\{1, 2, \ldots, 2N\}$. For convenience, we choose the following ordering:
\begin{equation}\label{eqn::linkpattern_order}
a_s<b_s,\text{ for all }s\in \{1, \ldots, N\}, \quad\text{and }a_1<a_2<\cdots<a_N. 
\end{equation}
We denote by $\LP_N$ the set of such planar link patterns. 
Note that $\#\LP_N$ is given by $N$:th Catalan number $\frac{1}{N+1}\binom{2N}{N}$. 
For each $\alpha=\{\{a_1, b_1\}, \ldots, \{a_N, b_N\}\}\in\LP_N$, we denote by $\chamber_{\alpha}(\Omega; x_1, \ldots, x_{2N})$ the collection of curves $(\eta^{(1)}, \ldots, \eta^{(N)})$ such that, for each $j\in\{1, \ldots, N\}$, the curve $\eta^{(j)}$ is a continuous non-self-crossing curves in $\Omega$ connecting $x_{a_j}$ and $x_{b_j}$ and $\eta^{(j)}$ does not disconnect any two points $x_{a_s}$ and $x_{b_s}$ for $s\neq j$. 
\begin{definition}
Global $N$-$\SLE_{\kappa}$ associated to the link pattern $\alpha\in\LP_N$ in the polygon $(\Omega; x_1, \ldots, x_{2N})$ is the unique probability measure on $\chamber_{\alpha}(\Omega; x_1, \ldots, x_{2N})$ with resampling property: for each $j\in\{1, \ldots, N\}$, the conditional law of the curve $\eta_j$ given $\{\eta^{(1)}, \ldots, \eta^{(N)}\}\setminus\{\eta^{(j)}\}$ is chordal $\SLE_{\kappa}$ connecting $x_{a_j}$ and $x_{b_j}$ in the connected component of the domain $\Omega\setminus\cup_{s\neq j}\eta^{(s)}$ having the end points $x_{a_j}$ and $x_{b_j}$ on its boundary.  
We denote its law by 
\[\QQ_{\alpha}=\QQ_{\alpha}(\Omega; x_1, \ldots, x_{2N}).\]
\end{definition}
There is extensive literature about the existence and uniqueness of global $N$-$\SLE_{\kappa}$ for different ranges of $\kappa$: \cite{KozdronLawlerMultipleSLEs, MillerSheffieldIG1, MillerSheffieldIG2, PeltolaWuGlobalMultipleSLEs, WuHyperSLE, BeffaraPeltolaWuUniqueness, AngHoldenSunYu2023, ZhanExistenceUniquenessMultipleSLE, FengLiuPeltolaWu2024}. 
They guarantee the existence and uniqueness of global $N$-$\SLE_{\kappa}$ for $\kappa\in (0,8)$. 
Its law is encoded by pure partition functions (see Section~\ref{subsec::pre_ppf}) 
\[\LZ_{\alpha}(\Omega; x_1, \ldots, x_{2N}).\]
In particular, when $\Omega=\HH$, they are solutions to chordal BPZ equations~\eqref{eqn::chordalBPZ}. 

\begin{proposition}\label{prop::solutions_evenr2}
Fix $N\ge 1, \alpha\in\LP_N$ and a polygon with an interior point $(\Omega; x_1, \ldots, x_{2N}; z)$. Fix $\kappa\in (0,6]$ and suppose $(\eta^{(1)}, \ldots, \eta^{(N)})\sim \QQ_{\alpha}$ is the global $N$-$\SLE_{\kappa}$ associated to $\alpha$ in $(\Omega; x_1, \ldots, x_{2N})$. 
We write $\bs{\eta}=\cup_{j=1}^N\eta^{(j)}$, 
and denote by $\CR\left(\Omega\setminus\bs{\eta}; z\right)$ the conformal radius of the connected component of $\Omega\setminus\bs{\eta}$ containing $z$. For $\mathfrak{r}\in\R$, we define 
\begin{equation}\label{eqn::solutions_evenr2}
\LZalphar(\Omega; x_1, \ldots, x_{2N}; z)=\LZ_{\alpha}(\Omega; x_1, \ldots, x_{2N})\E_{\alpha}\left[\CR\left(\Omega\setminus\bs{\eta}; z\right)^{-\mathfrak{r}}\right]. 
\end{equation}
Then the expectation in~\eqref{eqn::solutions_evenr2} is finite when $\mathfrak{r}<1-\kappa/8$. 
Furthermore, for $(\theta_1, \ldots, \theta_{2N})\in\LX_{2N}$, if we write 
\begin{equation}\label{eqn::LZalphar_radialcoordinates}
\LZalphar(\theta_1, \ldots, \theta_{2N})=\LZalphar(\U; \exp(\ii\theta_1), \ldots, \exp(\ii\theta_{2N});0),
\end{equation}
then $\LZalphar: \LX_{2N}\to \R_{>0}$ satisfies the system of radial BPZ equations~\eqref{eqn::radialBPZ} with
$\aleph=\frac{(6-\kappa)(\kappa-2)}{8\kappa}-\mathfrak{r}>-\frac{3}{2\kappa}$: for all $1\le j\le 2N$, 
\begin{equation}\label{eqn::radialBPZ_evenr2}
	\frac{\kappa}{2}\frac{\partial_{j}^2\LZalphar}{\LZalphar}+\sum_{\ell\neq j}\left(\cot((\theta_{\ell}-\theta_j)/2)\frac{\partial_{\ell}\LZalphar}{\LZalphar}-\frac{(6-\kappa)/(4\kappa)}{\left(\sin((\theta_{\ell}-\theta_j)/2)\right)^2}\right)=\frac{(6-\kappa)(\kappa-2)}{8\kappa}-\mathfrak{r}. 
\end{equation}
\end{proposition}
Proposition~\ref{prop::solutions_evenr2} with $N=1$ is proved in~\cite{WangWuCommutationRelation}. If we set $\mathfrak{r}=0$, then $\LZalphar$ in Proposition~\ref{prop::solutions_evenr2} becomes pure partition function $\LZ_{\alpha}$ and radial BPZ equations~\eqref{eqn::radialBPZ} with $\aleph=\frac{(6-\kappa)(\kappa-2)}{8\kappa}$ is equivalent to the chordal BPZ equations~\eqref{eqn::chordalBPZ} written in radial coordinates~\eqref{eqn::LZalphar_radialcoordinates}. 

\subsection{Relation between Theorem~\ref{thm::FKIsing_cvg} and Proposition~\ref{prop::solutions_evenr2}}
\label{subsec::intro_strategy}
For the FK-Ising model, we have the following asymptotic behavior of probabilities of one-arm events.
\begin{lemma} \label{lem::limit_one_point}
Fix a bounded simply connected domain $\Omega$  and suppose that a sequence of admissible medial domains $\Omega^{\delta,\diamond}$  converges to $\Omega$ in the Carath\'{e}odory sense (see details in Section~\ref{subsec::pre_rcm}). Suppose that $z^{\delta}\to z$ as $\delta\to 0$. Consider the critical FK-Ising model on the primal domain $\Omega^{\delta}$ with the wired boundary condition. Then we have
	\begin{equation} \label{eqn::limit_one_point}
\lim_{\delta\to 0 }\frac{\mathbb{P}^{\delta}_{\Omega,\mathrm{w}}[
	\exists\text{ an open path connecting } z^{\delta} \text{ to }\partial\Omega^{\delta}]}{\mathbb{P}^{\delta}_{\mathbb{U},\mathrm{w}}[\exists\text{ an open path connecting } 0 \text{ to }\partial\mathbb{U}^{\delta}]}= \CR(\Omega;z)^{-1/8}.
	\end{equation}
\end{lemma}
Indeed, it follows from~\cite[Theorem~1.2]{ChelkakHonglerIzyurovConformalInvarianceCorrelationIsing} and the Edwards-Sokal coupling that
\begin{equation} \label{eqn::asy_one_arm}
	\mathbb{P}^{\delta}_{\Omega,\mathrm{w}} [\exists\text{ an open path connecting } z^{\delta} \text{ to }\partial\Omega^{\delta}]=C\delta^{1/8}\CR(\Omega;z)^{-1/8}(1+o(1)),\quad \text{as }\delta\to 0, 
\end{equation} 
for some $C\in (0,\infty)$, which is stronger than Lemma~\ref{lem::limit_one_point}. We will give an alternative proof of Lemma~\ref{lem::limit_one_point} in Appendix~\ref{app::technical} based on the $\mathrm{CLE}$-convergence result stated in Proposition~\ref{prop::CLE} for the following two reasons: first, the weaker result Lemma~\ref{lem::limit_one_point} is sufficient for us to prove Theorem~\ref{thm::FKIsing_cvg}; second, while generalizing the arguments in~\cite{ChelkakHonglerIzyurovConformalInvarianceCorrelationIsing} to other critical models seems to be out of reach\footnote{One does not have an analogue of~\eqref{eqn::asy_one_arm} for percolation, for example.}, the recent groundbreaking work~\cite{DuminilCopinKozlowskiKrachunManolescuOulamaraRotationInvariance} suggests that a proof of the $\mathrm{CLE}$-convergence for critical random-cluster models with $q\in [1,4)$ other than $2$ may be more realistic. 

Let us explain the reason for radial BPZ equations~\eqref{eqn::radialBPZ_onearm} in Theorem~\ref{thm::FKIsing_cvg}. Assume the same setup as in Theorem~\ref{thm::FKIsing_cvg}, without conditioning, the law of the scaling limit $\bs{\eta}=(\eta^{(1)}, \ldots, \eta^{(N)})$ of the collection of interfaces $(\eta^{(1, \delta)}, \ldots, \eta^{(N, \delta)})$ is a linear combination of global $N$-$\SLE_{16/3}$ (due to~\cite{BeffaraPeltolaWuUniqueness} and~\cite{FengPeltolaWuConnectionProbaFKIsing}, see details in Section~\ref{subsec::con_inv_FK}). Given the collection of interfaces, the conditional probability of the one-arm event $\LA^{\delta}$ in~\eqref{eqn::onearm_event_def}, after proper normalization as in Lemma~\ref{lem::limit_one_point}, converges to $\CR(\Omega\setminus\bs{\eta}; z)^{-1/8}$. Therefore, the scaling limit of
the collection of interfaces $(\eta^{(1, \delta)}, \ldots, \eta^{(N, \delta)})$ is a linear combination of global $N$-$\SLE_{16/3}$ weighted by $\CR(\Omega\setminus\bs{\eta}; z)^{-1/8}$ (see details in Section~\ref{subsec::FKIsing_proof}). Combining with Proposition~\ref{prop::solutions_evenr2}, we find that the corresponding partition function $\LG^{(\mathfrak{r}_1)}$ satisfies radial BPZ equations~\eqref{eqn::radialBPZ} with 
\[\aleph=\frac{(6-\kappa)(\kappa-2)}{8\kappa}-\mathfrak{r}_1=\frac{16-\kappa^2}{32\kappa}\]
as in~\eqref{eqn::radialBPZ_onearm} for $\kappa=16/3$. 

\begin{remark}\label{rem::rcm}
Our proof for Theorem~\ref{thm::FKIsing_cvg} relies on three inputs: Proposition~\ref{prop::solutions_evenr2}, Lemma~\ref{lem::limit_one_point} and previous results of the convergence of the collection of FK-Ising interfaces (without conditioning). Conjecture~\ref{conj::cvg_interfaces} can be proved using the same strategy as long as one knows the convergence of a single interface to SLE and the convergence of the collection of loops to CLE. In particular, Conjecture~\ref{conj::cvg_interfaces} holds for Bernoulli site percolation on the triangular lattice with $\kappa=6$, as the convergence to $\SLE_6$ and $\CLE_6$ are known~\cite{SmirnovPercolationConformalInvariance, LawlerSchrammWernerOneArmExponent, CamiaNewmanPercolation, CamiaNewmanPercolationFull}. 
\end{remark}

\paragraph*{Outline.}
We will prove Proposition~\ref{prop::solutions_evenr2} in Section~\ref{sec::solutions}, prove Theorem~\ref{thm::FKIsing_cvg} in Section~\ref{sec::FKIsing} and prove Lemma~\ref{lem::limit_one_point} in Appendix~\ref{app::technical}. 

\paragraph*{Acknowledgments.}
We thank Federico Camia and Yilin Wang for helpful discussions. 
H.W. is partly affiliated with Yanqi Lake Beijing Institute of Mathematical Sciences and Applications, Beijing, China.

\section{Global multiple SLEs and proof of Proposition~\ref{prop::solutions_evenr2}}
\label{sec::solutions}
Fix parameters
\begin{equation}\label{eqn::htildeh_def}
\kappa\in (0,6],\qquad h=\frac{6-\kappa}{2\kappa},\qquad\tilde{h}=\frac{(6-\kappa)(\kappa-2)}{8\kappa}. 
\end{equation}
The goal of this section is to prove Proposition~\ref{prop::solutions_evenr2}. To this end, we first introduce the Poisson kernel and notations with radial Loewner chain. 

\paragraph*{Poisson kernel.}
(Boundary) Poisson kernel $H(\Omega; x, y)$ is defined for nice Dobrushin domain $(\Omega; x, y)$. When $\Omega=\HH$, we have
\begin{equation*}
H(\HH; x,y)=|y-x|^{-2},\quad x,y\in\R. 
\end{equation*}
For nice Dobrushin domain $(\Omega; x, y)$, we extend its definition via conformal covariance: 
\begin{equation*}
H(\Omega; x, y)=|\varphi'(x)\varphi'(y)|H(\HH; \varphi(x), \varphi(y)), 
\end{equation*}
where $\varphi$ is any conformal map from $\Omega\to \HH$.
When $\Omega=\U$, we have
\begin{equation*}
H(\U; \exp(\ii\theta_1), \exp(\ii\theta_2))=\left(2\sin((\theta_2-\theta_1)/2)\right)^{-2}.
\end{equation*}

Poisson kernel satisfies the following monotonicity.
Let $(\Omega; x, y)$ be a nice Dobrushin domain and let $U\subset\Omega$ be a simply connected domain that agrees with $\Omega$ in neighborhoods of $x$ and of $y$. Then we have
\begin{equation}\label{eqn::Poisson_mono}
H(U; x, y)\le H(\Omega; x, y).
\end{equation}

\paragraph*{Radial Loewner chain.}
Suppose $K\subset\U$ is a compact set such that $\U\setminus K$ is simply connected and contains the origin. Let $\varphi$ be the conformal map from $\U\setminus K$ onto $\U$ with $\varphi(0)=0$ and $\varphi'(0)>0$. 
The capacity of $K$ is $\log \varphi'(0)=-\log\CR(\U\setminus K;0)$. 

Fix $\theta\in [0,2\pi)$. Suppose $\eta:[0,T]\to \overline{\U}$ is a continuous non-self-crossing curve such that $\eta_0=\exp(\ii \theta)$. Let $U_t$ be the connected component of $\U\setminus \eta_{[0,t]}$ containing the origin. Let $g_t:U_t \to \U$ be the unique conformal map with $g_t(0)=0$ and $g'_t(0)>0$. We say that the curve is parameterized by capacity if $g'_t(0)=\exp(t)$. Then $g_t$ satisfies the radial Loewner chain:
\begin{equation*}
	\partial_t g_t(z)=g_t(z) \frac{\exp(\ii \xi_t)+g_t(z)}{\exp(\ii \xi_t)-g_t(z)}, \quad g_0(z)=z,
\end{equation*}
where $t\mapsto \xi_t \in \R$ is continuous and called the driving function of $\eta$. 
Radial $\SLE_{\kappa}$ is the radial Loewner chain with driving function $\xi_t=\sqrt{\kappa}B_t$ where $(B_t, t\ge 0)$ is one-dimensional Brownian motion. We will call it radial $\SLE_{\kappa}$ in $(\U; \exp(\ii\theta); 0)$ for short. Radial SLE in a general domain is defined via conformal image. 

Let $\phi_t$ be the covering map of $g_t$, i.e., the continuous function such that $g_t(\exp(\ii \theta))=\exp(\ii \phi_t(\theta))$ and $\phi_0(\theta)=\theta$, we have
$\partial_t \phi_t(\theta) =\cot \left( \left( \phi_t(\theta)-\xi_t \right)/2 \right)$.
In the following, we will have multiple marked points. For $(\theta_1, \ldots,\theta_p)\in\LX_p$, suppose $\eta:[0,T]\to \overline{\U}$ is a continuous non-self-crossing curve such that $\eta_0=\exp(\ii \theta_1)$. Then the evolution $V_t^{(j)}$ of the marked point $\theta_j$  is the same as $\phi_t(\theta_j)$ for $j\in \{2, \ldots, p\}$.

\medbreak
The proof of Proposition~\ref{prop::solutions_evenr2} is split into the following two lemmas. 
\begin{lemma}\label{lem::finite_expectation}
Assume the same notations as in Proposition~\ref{prop::solutions_evenr2}. The expectation in~\eqref{eqn::solutions_evenr2} is finite. Moreover, we have the following upper bound: there exists a constant $\cst\in (0,\infty)$ depending on $\kappa$ and $\mathfrak{r}$ such that 
\begin{equation}\label{eqn::LZalphar_upperbound}
\LZalphar(\Omega; x_1, \ldots, x_{2N}; z)\le N \cst \CR(\Omega; z)^{-\mathfrak{r}}\prod_{\{a,b\}\in\alpha} H(\Omega; x_a, x_b)^{h}. 
\end{equation}
We denote by 
\[\QQalphar=\QQalphar(\Omega; x_1, \ldots, x_{2N}; z)\]
the probability measure of $(\eta^{(1)}, \ldots, \eta^{(N)})\sim\QQ_{\alpha}$ global $N$-$\SLE_{\kappa}$ associated to $\alpha$ in polygon $(\Omega; x_1, \ldots, x_{2N})$ weighted by 
\begin{equation*}
\frac{\LZ_{\alpha}(\Omega; x_1, \ldots, x_{2N})}{\LZalphar(\Omega; x_1, \ldots, x_{2N}; z)}\CR(\Omega\setminus\bs{\eta}; z)^{-\mathfrak{r}}. 
\end{equation*}
\end{lemma}

\begin{lemma}\label{lem::marginal_LZalphar}
Assume the same notations as in Proposition~\ref{prop::solutions_evenr2} and in Lemma~\ref{lem::finite_expectation}. 
\begin{itemize}
\item The function $\LZalphar: \LX_{2N}\to \R_{>0}$ in~\eqref{eqn::LZalphar_radialcoordinates} satisfies the system of radial BPZ equations~\eqref{eqn::radialBPZ_evenr2}. 
\item For $(\theta_1, \ldots, \theta_{2N})\in\LX_{2N}$, suppose $(\eta^{(1)}, \ldots, \eta^{(N)})\sim\QQalphar$ in $(\U; \exp(\ii\theta_1), \ldots, \exp(\ii\theta_{2N}); 0)$. Then the law of $\eta^{(1)}$ under $\QQalphar$ is the same as radial $\SLE_{\kappa}$ in $(\U; \exp(\ii\theta_1); 0)$ weighted by the following local martingale, up to the first time $\exp(\ii\theta_{2})$ or $\exp(\ii\theta_{2N})$ is disconnected from the origin: 
\begin{equation}\label{eqn::mart_multipleSLE_CR}
M_t(\LZalphar)=g_t'(0)^{\mathfrak{r}-\tilde{h}}\times\prod_{j=2}^{2N}\phi_t'(\theta_j)^h\times\LZalphar(\xi_t, \phi_t(\theta_2), \ldots, \phi_t(\theta_{2N})). 
\end{equation}  
\end{itemize}
\end{lemma}

The rest of this section is organized as follows. We first give preliminaries on chordal SLE in Section~\ref{subsec::pre_SLE} and give preliminaries on pure partition functions in Section~\ref{subsec::pre_ppf}. Then we prove Lemma~\ref{lem::finite_expectation} in Section~\ref{subsec::finite_expectation} and prove Lemma~\ref{lem::marginal_LZalphar} in Section~\ref{subsec::marginal_LZalphar}. Finally, we give a generalization of Proposition~\ref{prop::solutions_evenr2} in Section~\ref{subsec::solutions_wiredbc}, the purpose of such generalization will be clear in the proof of Theorem~\ref{thm::FKIsing_cvg} in Section~\ref{subsec::FKIsing_proof}. 

\subsection{Preliminaries on chordal SLE}
\label{subsec::pre_SLE}

Chordal SLE is usually defined in the upper-half plane, in this article, it is more convenient to write down its definition in the unit disc. Fix $(\theta_1, \theta_2)\in\LX_2$, the law of a chordal $\SLE_{\kappa}$ in $(\U; \exp(\ii\theta_1), \exp(\ii\theta_2))$ is the same as a radial $\SLE_{\kappa}$ in $(\U; \exp(\ii\theta_1); 0)$ weighted by the following local martingale, up to the first time $\exp(\ii\theta_2)$ is disconnected from the origin: 
\[M_t(\LZtwo):=g_t'(0)^{-\tilde{h}}\times\phi_t'(\theta_2)^h\times \LZtwo(\xi_t, \phi_t(\theta_2)), \quad \text{where }\LZtwo(\theta_1, \theta_2)=(2\sin((\theta_2-\theta_1)/2))^{-2h}.\]
We say that the partition function for chordal $\SLE_{\kappa}$ in $(\U; \exp(\ii\theta_1), \exp(\ii\theta_2))$ is $(2\sin((\theta_2-\theta_1)/2))^{-2h}$.
Chordal SLE in a general domain is defined via conformal image. Furthermore, the partition function for chordal $\SLE_{\kappa}$ in $(\Omega; x_1, x_2)$ is given by 
\begin{equation}\label{eqn::PPF_Nequalsone}
\LZtwo(\Omega; x_1, x_2)=H(\Omega; x_1, x_2)^h. 
\end{equation}

\begin{lemma}\label{lem::SLE_CR_expectation}
Fix $\kappa\in (0,8)$ and a Dobrushin domain with an interior point $(\Omega; x_1, x_2; z)$. Suppose $\eta\sim\QQtwo$ is chordal $\SLE_{\kappa}$ in $(\Omega; x_1, x_2)$. For $\mathfrak{r}\in\R$, define
\begin{equation}\label{eqn::ppftwor_def}
\LZtwor(\Omega; x_1, x_2; z)=\LZtwo(\Omega; x_1, x_2)\Etwo\left[\CR(\Omega\setminus\eta; z)^{-\mathfrak{r}}\right]. 
\end{equation}
Then the expectation in~\eqref{eqn::ppftwor_def} is finite if and only if $\mathfrak{r}<1-\kappa/8$. Furthermore, if we denote 
\begin{equation}\label{eqn::LZtwor_pi_def}
\LZtwor(\pi):=\LZtwor(\U; 1, -1; 0).
\end{equation}
Then, we have the following upper bound:
\begin{equation}\label{eqn::LZtwor_bound}
\LZtwor(\Omega; x_1, x_2; z)\le
\begin{cases}
\CR(\Omega; z)^{-\mathfrak{r}} H(\Omega; x_1, x_2)^h, &\text{if }\mathfrak{r}\le 0;\\
4^h\LZtwor(\pi)\CR(\Omega; z)^{-\mathfrak{r}} H(\Omega; x_1, x_2)^h, &\text{if }0<\mathfrak{r}<1-\kappa/8. 
\end{cases}
\end{equation}
\end{lemma} 
\begin{proof}
It suffices to show the conclusion for $\Omega=\U$ and $z=0$. 
The first part of the conclusion is proved in~\cite[Proof of Lemma~3.10]{WangWuCommutationRelation}. Furthermore, the proof there gives the following description of the expectation in $\U$. For $(\theta_1, \theta_2)\in\LX_2$, we write
\[\LZtwor(\theta_1, \theta_2)=\LZtwor(\U; \exp(\ii\theta_1), \exp(\ii\theta_2); 0).\]
We denote $\theta=\theta_2-\theta_1$ and write 
\[\LZtwor(\theta_1, \theta_2)=(2\sin(\theta/2))^{-2h}\Phi(\kappa, \mathfrak{r}; u), \quad\text{where }u=\sin^2(\theta/4).\]
When $\mathfrak{r}<1-\kappa/8$, the function $u\mapsto \Phi(\kappa, \mathfrak{r}; u)$ satisfies the following Euler's hypergeometric differential equation:
\begin{equation*}
    u(1-u)\Phi''+\frac{3\kappa-8}{2\kappa}(1-2u)\Phi'+\frac{8\mathfrak{r}}{\kappa}\Phi=0.
\end{equation*}
Then there are two cases.
\begin{itemize}
\item When $\mathfrak{r}\le 0$, as $\CR(\U\setminus\eta)\le 1$, we have
$
\LZtwor(\theta_1, \theta_2)\le \LZtwo(\theta_1, \theta_2)=(2\sin(\theta/2))^{-2h}
$
as desired in~\eqref{eqn::LZtwor_bound}.
\item When $0\le \mathfrak{r}<1-\kappa/8$, we have 
$\Phi(\kappa, \mathfrak{r}; u)\le \Phi(\kappa, \mathfrak{r}; 1/2)$.
Thus, 
\begin{align*}
\LZtwor(\theta_1, \theta_2)=&(2\sin(\theta/2))^{-2h}\Phi(\kappa, \mathfrak{r}; u)\\
\le& (2\sin(\theta/2))^{-2h}\Phi(\kappa, \mathfrak{r}; 1/2)\\
=& (\sin(\theta/2))^{-2h}\LZtwor(\pi)=4^h\LZtwor(\pi) (2\sin(\theta/2))^{-2h},
\end{align*}
as desired in~\eqref{eqn::LZtwor_bound}. Note that 
\begin{equation}\label{eqn::LZtworpi_Euler}
\Phi(\kappa, \mathfrak{r}; 1/2)=4^h\LZtwor(\pi). 
\end{equation}

\end{itemize}
\end{proof}


The following estimate will be useful in the proof of Lemma~\ref{lem::finite_expectation}.
\begin{corollary}\label{cor::CR_estimate}
Fix $\kappa\in (0,8)$ and suppose $\eta\sim\QQtwo$ is chordal $\SLE_{\kappa}$ in $(\U; \exp(\ii\theta_1), \exp(\ii\theta_2))$. 
For any $p\in (0, 1-\kappa/8)$ and $\eps\in (0,1)$, we have
\begin{align}\label{eqn::CR_estimate1}
\QQtwo\left[\CR(\U\setminus\eta)\le\eps\right]\le 4^h\LZtwo^{(p)}(\pi) \eps^{p}. 
\end{align}
\end{corollary}
\begin{proof}
We assume the same notations as in the proof of Lemma~\ref{lem::SLE_CR_expectation}. By Markov inequality, we have
\begin{align*}
\QQtwo\left[\CR(\U\setminus\eta)\le \eps\right]=&\QQtwo\left[\CR(\U\setminus\eta)^{-p}\ge \eps^{-p}\right]\\
\le& \eps^{p}\Etwo\left[\CR(\U\setminus\eta)^{-p}\right]=\eps^{p}\Phi(\kappa, p;u).
\end{align*}
From~\eqref{eqn::LZtworpi_Euler}, we have
$\Phi(\kappa, p;u)\le \Phi(\kappa, p;1/2)=4^h\LZtwo^{(p)}(\pi)$.
This completes the proof of~\eqref{eqn::CR_estimate1}. 
\end{proof}

\subsection{Preliminaries on pure partition functions}
\label{subsec::pre_ppf}
Recall that $\LP_N$ denotes the set of all planar link patterns among $2N$ boundary points. We denote $\LP=\sqcup_{N\ge 0}\LP_N$. 
Pure partition functions of multiple $\SLE_{\kappa}$ are the recursive collection $\{\LZ_{\alpha} \colon \alpha \in \LP\}$ 
of functions \[\LZ_{\alpha}(\HH; \cdot) \colon \{(x_1,\ldots,x_{2N})\in\R^{2N}: x_1<\cdots<x_{2N}\}\to\R\]
uniquely determined by the following four properties:
\begin{itemize}
	\item Chordal BPZ equations: for all $ j \in \{1,\ldots,2n\}$, 
	\begin{align*}
		\left[ 
		\frac{\kappa}{2} \partial_j^2
		+ \sum_{\ell\neq j} \left( \frac{2}{x_{\ell}-x_{j}}\partial_{\ell}
		- \frac{2h}{(x_{\ell}-x_{j})^{2}} \right) \right]
		\LZ_{\alpha}(\HH; x_1,\ldots,x_{2n}) =  0.
	\end{align*}
	\item M\"{o}bius covariance: for all M\"obius maps $\varphi$ of the upper half-plane $\HH$ such that $\varphi(x_{1}) < \cdots < \varphi(x_{2N})$, we have
	\begin{align*}
		\LZ_{\alpha}(\HH; x_{1},\ldots,x_{2N}) = 
		\prod_{j=1}^{2N} \varphi'(x_{j})^{h} 
		\times \LZ_{\alpha}(\HH; \varphi(x_{1}),\ldots,\varphi(x_{2N})).
	\end{align*}
	\item Asymptotics: with $\LZ_{\emptyset} \equiv 1$ for the empty link pattern $\emptyset \in \LP_0$, the collection $\{\LZ_{\alpha} \colon \alpha\in\LP\}$ satisfies the following recursive asymptotics property. Fix $N\ge 1$ and $j \in \{1,2, \ldots, 2N-1 \}$. 
	Then, we have
	\begin{align*}
		\lim_{x_j,x_{j+1}\to\xi} \frac{\LZ_{\alpha}(\HH; x_1,\ldots, x_{2N})}{ (x_{j+1}-x_j)^{-2h} }
		= 
		\begin{cases}
			\LZ_{\alpha/\{j,j+1\}}(\HH; x_1, \ldots, x_{j-1}, x_{j+2}, \ldots, x_{2N}), 
			& \quad \text{if }\{j, j+1\}\in\alpha , \\
			0 ,
			& \quad \text{if }\{j, j+1\} \not\in \alpha ,
		\end{cases}
	\end{align*}
	where $\xi \in (x_{j-1}, x_{j+2})$ (with the convention that $x_0 = -\infty$ and  $x_{2N+1} = +\infty$), and $\alpha/\{k,l\}$ denotes the link pattern in $\LP_{N-1}$ obtained by removing $\{k,l\}$ from $\alpha$ and then relabeling the remaining indices so that they are the first $2(N-1)$ positive integers. 
	\item The functions are positive and satisfy the following power-law bound:
	\begin{align}\label{eqn::PPF_PLB}
		0<\LZ_{\alpha}(\HH; x_1, \ldots, x_{2N})\le\prod_{\{a,b\}\in\alpha}|x_a-x_b|^{-2h}, \quad \text{for all }x_1<\cdots<x_{2N}. 
	\end{align}
\end{itemize}

The uniqueness when $\kappa\in (0,8)$ of such collection of functions is proved in~\cite{FloresKlebanPDE2}. The existence when $\kappa\in (0,6]$ is proved in~\cite{WuHyperSLE}. For other results related to the existence of such functions, see~\cite{FloresKlebanPDE3, KytolaPeltolaPurePartitionFunctions, PeltolaWuGlobalMultipleSLEs, AngHoldenSunYu2023, FengLiuPeltolaWu2024}. 
We extend the definition of $\LZ_{\alpha}$ to more general
polygons $(\Omega; x_1, \ldots, x_{2N})$ as 
\begin{align}\label{eqn::PartF_def_polygon}
	\LZ_{\alpha}(\Omega; x_1, \ldots, x_{2N}):=\prod_{j=1}^{2N}|\varphi'(x_j)|^h\times \LZ_{\alpha}(\HH; \varphi(x_1), \ldots, \varphi(x_{2N})),
\end{align}
where $\varphi$ is any conformal map from $\Omega$ onto $\HH$ with $\varphi(x_1)<\cdots<\varphi(x_{2N})$. 
The power-law bound~\eqref{eqn::PPF_PLB} becomes
\begin{equation}\label{eqn::PPF_PLB_polygon}
0<\LZ_{\alpha}(\Omega; x_1, \ldots, x_{2N})\le \prod_{\{a,b\}\in\alpha}H(\Omega; x_a, x_b)^{h}. 
\end{equation}
For the polygon $(\U; \exp(\ii\theta_1), \ldots, \exp(\ii\theta_{2N}))$ with $(\theta_1, \ldots, \theta_{2N})\in\LX_{2N}$, we write
\begin{align*}
	\LZ_{\alpha}\left(\theta_1, \ldots, \theta_{2N}\right)=\LZ_{\alpha}\left(\U; \exp\left(\ii\theta_1\right), \ldots, \exp\left(\ii\theta_{2N}\right)\right).
\end{align*}
Then the chordal BPZ equations~\eqref{eqn::chordalBPZ} become radial BPZ equations~\eqref{eqn::radialBPZ} with
$\aleph=\frac{(6-\kappa)(\kappa-2)}{8\kappa}$. 
\medbreak
The Loewner evolution in global multiple SLEs can be described by pure partition functions. 

\begin{lemma}\label{lem::globalnSLE_mart}
Fix $\kappa\in (0,6], N\ge 1$, $\alpha\in\LP_N$ and $(\theta_1, \ldots, \theta_{2N})\in\LX_{2N}$. 
Suppose $\bs{\eta}=(\eta^{(1)}, \ldots, \eta^{(N)})\sim\QQ_{\alpha}$ is global $N$-$\SLE_{\kappa}$ associated to $\alpha$ in $(\U; \exp(\ii\theta_1), \ldots, \exp(\ii\theta_{2N}))$. 
Then the law of $\eta^{(1)}$ under $\QQ_{\alpha}$ is the same as radial $\SLE_{\kappa}$ in $(\U; \exp(\ii\theta_1); 0)$ weighted by the following local martingale, up to the first time $\exp(\ii\theta_{2})$ or $\exp(\ii\theta_{2N})$ is disconnected from the origin:
\begin{equation}\label{eqn::globalnSLE_mart_cor}
M_t(\LZ_{\alpha})=g_t'(0)^{-\tilde{h}}\times\prod_{j=2}^{2N}\phi_t'(\theta_j)^h\times\LZ_{\alpha}(\xi_t, \phi_t(\theta_2), \ldots, \phi_t(\theta_{2N})). 
\end{equation}  
\end{lemma}
\begin{proof}
Analogous conclusion is known for the upper-half plane, see e.g.~\cite[Section~6]{WuHyperSLE}. From $\HH$ to $\U$, we perform a standard change of variables calculation~\cite{SchrammWilsonSLECoordinatechanges}. 
\end{proof}

The following cascade relation for pure partition functions will be useful later. 
Suppose $(\eta^{(1)}, \ldots, \eta^{(N)})\sim\QQ_{\alpha}$ is global $N$-$\SLE_{\kappa}$ associated to $\alpha$ in polygon $(\Omega; x_1, \ldots, x_{2N})$. 
The curve $\eta^{(1)}$ is in $\Omega$ from $x_{a_1}$ to $x_{b_1}$. We will describe the Radon-Nikodym derivative between $\eta^{(1)}$ under $\QQ_{\alpha}$ and chordal $\SLE_{\kappa}$ in $(\Omega; x_{a_1}, x_{b_1})$ below. 
For $\alpha=\{\{a_1, b_1\}, \ldots, \{a_N, b_N\}\}$, the link $\{a_1, b_1\}$ divides $\alpha$ into two sub-link patterns, connecting $\{a_1+1, \ldots, b_1-1\}$ and $\{b_1+1, \ldots, a_1-1\}$ respectively. After relabeling the indices, we denote these two link patterns by $\alpha_1^R$ and $\alpha_1^L$. 
Suppose $\eta$ is $\SLE_{\kappa}$ in $\Omega$ from $x_{a_1}$ to $x_{b_1}$, we say that $\eta$ is allowed by $\alpha$ if, for all $s\neq 1$, the points $x_{a_s}$ and $x_{b_s}$ lie on the boundary of the same connected component of $\U\setminus\eta$. In other words, $\eta$ is allowed by $\alpha$ if it does not disconnect any pair of points $\{x_{a_s}, x_{b_s}\}$ for $s\neq 1$. We denote this event by $\LE_{\alpha}(\eta)$. 
On the event $\LE_{\alpha}(\eta)$, the points $x_{a_1+1}, \ldots, x_{b_1-1}$ are divided into smaller groups. We denote the connected components of $\Omega\setminus\eta$ having these points on the boundary by $\Omega_1^{R, 1}, \ldots, \Omega_1^{R, r}$ in counterclockwise order and denote their union by $\Omega_1^R$. The sub-link pattern $\alpha_1^R$ is further divided into smaller sub-link patterns, after relabeling the indices, we denote these link patterns by $\alpha_1^{R, 1}, \ldots, \alpha_1^{R, r}$. We define 
\[\LZ_{\alpha_1^R}(\Omega_1^R; x_{a_1+1}, \ldots, x_{b_1-1})=\LZ_{\alpha_1^{R,1}}(\Omega_1^{R,1}; \ldots)\times\cdots\times\LZ_{\alpha_1^{R, r}}(\Omega_1^{R, r}; \ldots).\]
We define $\Omega_1^{L, 1}, \ldots, \Omega_1^{L, \ell}, \Omega_1^L, \alpha_1^{L, 1}, \ldots, \alpha_1^{L,\ell}$ similarly and define 
\[\LZ_{\alpha_1^L}(\Omega_1^L; x_{b_1+1}, \ldots, x_{a_1-1})=\LZ_{\alpha_1^{L,1}}(\Omega_1^{L,1}; \ldots)\times\cdots\times\LZ_{\alpha_1^{L, \ell}}(\Omega_1^{L, \ell}; \ldots).\]
We also write 
\begin{align*}
&\LZ_{\alpha/\{a_1, b_1\}}(\Omega\setminus\eta; x_{a_1+1}, \ldots, x_{b_1-1}, x_{b_1+1}, \ldots, x_{a_1-1})\\
=&\LZ_{\alpha_1^R}(\Omega_1^R; x_{a_1+1}, \ldots, x_{b_1-1})\times \LZ_{\alpha_1^L}(\Omega_1^L; x_{b_1+1}, \ldots, x_{a_1-1}).
\end{align*}
\begin{lemma}\label{lem::multipleSLEvschordalSLE}
Fix $\kappa\in (0,6], N\ge 1$, $\alpha\in\LP_N$ and a polygon $(\Omega; x_1, \ldots, x_{2N})$. Suppose $\eta\sim\QQtwo$ is chordal $\SLE_{\kappa}$ in $(\Omega; x_{a_1}, x_{b_1})$. 
Pure partition functions have the following cascade relation: 
\begin{align*}
\LZ_{\alpha}(\Omega; x_1, \ldots, x_{2N})=\LZtwo(\Omega; x_{a_1}, x_{b_1})\Etwo\left[\LZ_{\alpha/\{a_1, b_1\}}(\Omega\setminus\eta; x_{a_1+1}, \ldots, x_{b_1-1}, x_{b_1+1}, \ldots, x_{a_1-1})\one\{\LE_{\alpha}(\eta)\}\right]
\end{align*}
Furthermore, the law of $\eta^{(1)}$ under $\QQ_{\alpha}(\Omega; x_1, \ldots, x_{2N})$ is the same as $\eta\sim\QQtwo(\Omega; x_{a_1}, x_{b_1})$ weighted by 
\[\frac{\LZtwo(\Omega; x_{a_1}, x_{b_1}) }{\LZ_{\alpha}(\Omega; x_1, \ldots, x_{2N})}\LZ_{\alpha/\{a_1, b_1\}}(\Omega\setminus\eta; x_{a_1+1}, \ldots, x_{b_1-1}, x_{b_1+1}, \ldots, x_{a_1-1})\one\{\LE_{\alpha}(\eta)\}.\]
\end{lemma}
\begin{proof}
See~\cite[Section~6]{WuHyperSLE}.
\end{proof}

\subsection{Proof of Lemma~\ref{lem::finite_expectation}}
\label{subsec::finite_expectation}

\begin{proof}[Proof of Lemma~\ref{lem::finite_expectation}]
It suffices to show the conclusion for $\Omega=\U$ and $z=0$.
Recall that $\bs{\eta}=(\eta^{(1)}, \ldots, \eta^{(N)})\sim\QQ_{\alpha}$ is global $N$-$\SLE_{\kappa}$ associated to $\alpha$ in polygon $(\U; x_1, \ldots, x_{2N})$. We denote $\bs{\eta}=\cup_{j=1}^N\eta^{(j)}$ and denote by $\CR(\U\setminus\bs{\eta})$ the conformal radius of the connected component of $\U\setminus\bs{\eta}$ containing the origin. Note that $\CR(\U\setminus\bs{\eta})\le 1$. 
If $\mathfrak{r}\le 0$, we have
\begin{align*}
\LZalphar(\U; x_1, \ldots, x_{2N}; 0)\le &\LZ_{\alpha}(\U; x_1, \ldots, x_{2N})
\le\prod_{\{a,b\}\in\alpha}H(\U; x_a, x_b)^h,\tag{due to~\eqref{eqn::PPF_PLB_polygon}}
\end{align*}
as desired in~\eqref{eqn::LZalphar_upperbound}. 

In the rest of the proof, we assume $\mathfrak{r}>0$.
We write $\alpha\in\LP_N$ as in~\eqref{eqn::linkpattern_order}. 
Let us estimate the probability $\QQ_{\alpha}\left[\CR\left(\U\setminus\bs{\eta}\right)\le \eps\right]$ for $\eps>0$ small. 
For any subset $K\subset\U$, Koebe's one quarter theorem gives 
$\dist(0,K)\le \CR(\U\setminus K)\le 4\dist(0,K)$.
Thus,
\begin{align*}
\left\{\CR\left(\U\setminus\bs{\eta}\right)\le \eps\right\}
\quad \subset\quad 
\cup_{j=1}^N\left\{\dist(0,\eta^{(j)})\le \eps\right\}
\quad \subset\quad
\cup_{j=1}^N\left\{\CR(\U\setminus\eta^{(j)})\le 4\eps\right\}. 
\end{align*}
Consequently,
\begin{align}\label{eqn::multipleSLE_CR_aux2}
\QQ_{\alpha}\left[\CR\left(\U\setminus\bs{\eta}\right)\le \eps\right]\le \sum_{j=1}^N\QQ_{\alpha}\left[\CR(\U\setminus\eta^{(j)})\le 4\eps\right].
\end{align}
It suffices to estimate $\QQ_{\alpha}[\CR(\U\setminus\eta^{(j)})\le 4\eps]$.  

From Lemma~\ref{lem::multipleSLEvschordalSLE}, the law of $\eta^{(j)}$ under $\QQ_{\alpha}$ is absolutely continuous with respect to $\eta\sim\QQtwo$ chordal $\SLE_{\kappa}$ in $(\U; x_{a_j}, x_{b_j})$. We denote by $\LE^j_{\alpha}(\eta)$ the event that $\eta$ is allowed by $\alpha$. Then the law of $\eta^{(j)}$ under $\QQ_{\alpha}$ is the same as $\eta$ weighted by 
\[\frac{\LZtwo(\U; x_{a_j}, x_{b_j}) }{\LZ_{\alpha}(\U; x_1, \ldots, x_{2N})}\LZ_{\alpha/\{a_j,b_j\}}(\U\setminus\eta; x_{a_j}+1, \ldots, x_{b_j-1}, x_{b_j+1}, \ldots, x_{a_j-1})\one\left\{\LE^j_{\alpha}(\eta)\right\}.\]
From~\eqref{eqn::PPF_PLB_polygon} and the monotonicity of Poisson kernel~\eqref{eqn::Poisson_mono}, we have  
\begin{align}\label{eqn::multipleSLE_CR_aux1}
\LZ_{\alpha/\{a_j,b_j\}}(\U\setminus\eta; x_{a_j}+1, \ldots, x_{b_j-1}, x_{b_j+1}, \ldots, x_{a_j-1})
\le  \prod_{s\neq j}H(\U; x_{a_s}, x_{b_s})^h. 
\end{align}
We pick $p\in (0,1-\kappa/8)$, then 
\begin{align*}
&\QQ_{\alpha}[\CR(\U\setminus\eta^{(j)})\le 4\eps]\\
=&\frac{\LZ_{\includegraphics[scale=0.2]{figures/link-0}}(\U; x_{a_j}, x_{b_j}) }{\LZ_{\alpha}(\U; x_1, \ldots, x_{2N})}\Etwo\left[\LZ_{\alpha/\{a_j,b_j\}}(\U\setminus\eta; x_{a_j}+1, \ldots, x_{b_j-1}, x_{b_j+1}, \ldots, x_{a_j-1})\one\left\{\LE_{\alpha}^j(\eta)\cap\{\CR(\U\setminus\eta)\le 4\eps\}\right\}\right]\\
\le & \frac{\prod_{s=1}^N H(\U; x_{a_s}, x_{b_s})^h}{\LZ_{\alpha}(\U; x_1, \ldots, x_{2N})}\QQtwo[\CR(\U\setminus\eta)\le 4\eps]\tag{due to~\eqref{eqn::multipleSLE_CR_aux1}}\\
\le & \frac{\prod_{s=1}^N H(\U; x_{a_s}, x_{b_s})^h}{\LZ_{\alpha}(\U; x_1, \ldots, x_{2N})} 4^h \LZtwo^{(p)}(\pi) (4\eps)^p. \tag{due to~\eqref{eqn::CR_estimate1}}
\end{align*}
Plugging into~\eqref{eqn::multipleSLE_CR_aux2}, we obtain:
\begin{align}\label{eqn::multipleSLE_CR_aux0}
\QQ_{\alpha}\left[\CR\left(\U\setminus\bs{\eta}\right)\le \eps\right]\le N 4^p 4^h \LZtwo^{(p)}(\pi) \frac{\prod_{s=1}^N H(\U; x_{a_s}, x_{b_s})^h}{\LZ_{\alpha}(\U; x_1, \ldots, x_{2N})} \eps^p. 
\end{align}

Now, we are ready to show the conclusion. For $\mathfrak{r}\in (0, 1-\kappa/8)$, we pick $p\in (\mathfrak{r}, 1-\kappa/8)$, then we have
\begin{align*}
\E_{\alpha}\left[\CR\left(\U\setminus\bs{\eta}\right)^{-\mathfrak{r}}\right]\le &\sum_{k=1}^{\infty}2^{k\mathfrak{r}}\QQ_{\alpha}\left[2^{-k}\le \CR\left(\U\setminus\bs{\eta}\right)\le 2^{-k+1}\right]\\
\le & N 4^p 4^h \LZtwo^{(p)}(\pi) \frac{\prod_{s=1}^N H(\U; x_{a_s}, x_{b_s})^h}{\LZ_{\alpha}(\U; x_1, \ldots, x_{2N})}\sum_{k=1}^{\infty}2^{k\mathfrak{r}}2^{-(k-1)p}\tag{due to~\eqref{eqn::multipleSLE_CR_aux0}}\\
=&N 4^p 4^h \LZtwo^{(p)}(\pi) \frac{\prod_{s=1}^N H(\U; x_{a_s}, x_{b_s})^h}{\LZ_{\alpha}(\U; x_1, \ldots, x_{2N})}\frac{2^{p}}{1-2^{\mathfrak{r}-p}}.
\end{align*}
Therefore, 
\begin{align*}
\LZalphar(\U; x_1, \ldots, x_{2N}; 0)\le N 4^p 4^h \LZtwo^{(p)}(\pi) \frac{2^{p}}{1-2^{\mathfrak{r}-p}}\times \prod_{s=1}^N H(\U; x_{a_s}, x_{b_s})^h.
\end{align*}
This gives~\eqref{eqn::LZalphar_upperbound} by choosing $p\in (\mathfrak{r}, 1-\kappa/8)$ and setting
\begin{align*}
\cst=4^p 4^h \LZtwo^{(p)}(\pi) \frac{2^{p}}{1-2^{\mathfrak{r}-p}}. 
\end{align*}
\end{proof}

\subsection{Proof of Lemma~\ref{lem::marginal_LZalphar}}
\label{subsec::marginal_LZalphar}
\begin{proof}[Proof of Lemma~\ref{lem::marginal_LZalphar}]
For $(\theta_1, \ldots, \theta_{2N})\in\LX_{2N}$, suppose $\bs{\eta}=(\eta^{(1)}, \ldots, \eta^{(N)})\sim \QQ_{\alpha}$ is global $N$-$\SLE_{\kappa}$ associated to $\alpha$ in $(\U; \exp(\ii\theta_1), \ldots, \exp(\ii\theta_{2N}))$. 
We denote $\bs{\eta}=\cup_{j=1}^N\eta^{(j)}$ and denote by $\CR(\U\setminus\bs{\eta})$ the conformal radius of the connected component of $\U\setminus\bs{\eta}$ containing the origin.
We denote by $\QQalphar$
the law of $\QQ_{\alpha}$ weighted by \[\frac{\LZ_{\alpha}(\U; \exp(\ii\theta_1), \ldots, \exp(\ii\theta_{2N}))}{\LZalphar(\U; \exp(\ii\theta_1), \ldots, \exp(\ii\theta_{2N});0)}\CR\left(\U\setminus\bs{\eta}\right)^{-\mathfrak{r}}.\]
It suffices to show the conclusions for $j=1$. 

For $\eta=\eta^{(1)}$, let us calculate the conditional expectation of $\CR\left(\U\setminus\bs{\eta}\right)^{-\mathfrak{r}}$ given $\eta_{[0,t]}$ for $t>0$ small.
From the conformal invariance and domain Markov property of global $N$-$\SLE_{\kappa}$, we have
\begin{align}\label{eqn::mart_multipleSLE_aux1}
\E_{\alpha}\left[\CR\left(\U\setminus\bs{\eta}\right)^{-\mathfrak{r}}\cond \eta_{[0,t]}\right]=\ee^{\mathfrak{r}t}\frac{\LZalphar(\xi_t, \phi_t(\theta_2), \ldots, \phi_t(\theta_{2N}))}{\LZ_{\alpha}(\xi_t, \phi_t(\theta_2), \ldots, \phi_t(\theta_{2N}))}=\frac{M_t(\LZalphar)}{M_t(\LZ_{\alpha})}, 
\end{align}
where $M_t(\LZ_{\alpha})$ is defined in~\eqref{eqn::globalnSLE_mart_cor} and $M_t(\LZalphar)$ is defined in~\eqref{eqn::mart_multipleSLE_CR}. 
Combining Lemma~\ref{lem::globalnSLE_mart} and~\eqref{eqn::mart_multipleSLE_aux1}, the law of $\eta=\eta^{(1)}$ under $\QQalphar$ is the same as radial $\SLE_{\kappa}$ in $(\U; \exp(\ii\theta_1); 0)$ weighted by $M_t(\LZalphar)$. 

It remains to show that $\LZalphar$ satisfies radial BPZ equation~\eqref{eqn::radialBPZ_evenr2} with $j=1$. 
Let us calculate $\ud M_t(\LZalphar)$ assuming that $\LZalphar$ is $C^2$. Recall that
\begin{align*}
g_t'(0)=\ee^t,\quad &\partial_t \phi_t(w)=\cot((\phi_t(w)-\xi_t)/2),\quad 
\partial_t\phi_t'(w)=-\frac{1}{2}\csc^2((\phi_t(w)-\xi_t)/2)\phi_t'(w). 
\end{align*}
Thus, It\^{o}'s calculus gives
\begin{align}\label{eqn::mart_multipleSLE_aux2}
&\frac{\ud M_t(\LZalphar)}{M_t(\LZalphar)}\\
=&\frac{\partial_1\LZalphar}{\LZalphar}\ud \xi_t+\left[\frac{\kappa}{2}\frac{\partial_1^2\LZalphar}{\LZalphar}+\sum_{j=2}^{2N}\left(\cot\left((\phi_t(\theta_{j})-\xi_t)/2\right)\frac{\partial_j\LZalphar}{\LZalphar}-\frac{h}{2}\csc^2\left((\phi_t(\theta_{j})-\xi_t)/2\right)\right)+(\mathfrak{r}-\tilde{h})\right]\ud t.\notag
\end{align}

As $M_t(\LZalphar)$ is a local martingale for radial $\SLE_{\kappa}$, the second term in RHS of~\eqref{eqn::mart_multipleSLE_aux2} has to vanish. Thus $\LZalphar$ has to solve radial BPZ equation~\eqref{eqn::radialBPZ_evenr2} with $j=1$ under the assumption that $\LZalphar$ is $C^2$. Let us elaborate on the $C^2$ assumption here. In fact, the above analysis implies that $\LZalphar$ is a weak solution for~\eqref{eqn::radialBPZ_evenr2} with $j=1$ (see~\cite[Appendix~A]{FengLiuPeltolaWu2024}). As the operator in LHS of~\eqref{eqn::radialBPZ_evenr2} is hypoelliptic, weak solutions are strong solutions, thus $\LZalphar$ is indeed smooth and satisfies~\eqref{eqn::radialBPZ_evenr2} as desired. 
\end{proof}

\subsection{A generalization of Proposition~\ref{prop::solutions_evenr2}}
\label{subsec::solutions_wiredbc}

For $\kappa\in (0,6]$, suppose $\bs{\eta}=(\eta^{(1)}, \ldots, \eta^{(N)})\sim\QQ_{\alpha}$ is global $N$-$\SLE_{\kappa}$ associated to $\alpha$ in polygon $(\Omega; x_1, \ldots, x_{2N})$ and order $\alpha$ as in~\eqref{eqn::linkpattern_order}. We orient $\eta^{(j)}$ from $x_{a_j}$ to $x_{b_j}$. 
We denote by $\Omega_{\bs{\eta}}(z)$ the connected component of $\Omega\setminus\bs{\eta}$ containing $z$. 
We will give a generalization of Proposition~\ref{prop::solutions_evenr2}. 
We first explain the generalization when $\kappa\le 4$ as it is easier to describe in this case and then  give the formal definition for general $\kappa\in (0,6]$.
 
Fix $\kappa\in (0,4]$.  The $N$ curves $\eta^{(1)}, \ldots, \eta^{(N)}$ are disjoint and $\Omega\setminus\bs{\eta}$ has $(N+1)$ connected components. If we denote these $(N+1)$ connected components by $D_1, \ldots, D_{N+1}$ and replace $\E_{\alpha}\left[\CR(\Omega\setminus\bs{\eta};z)^{-\mathfrak{r}}\right]$ by $\E_{\alpha}\left[\one\{z\in D_j\}\CR(\Omega\setminus\bs{\eta};z)^{-\mathfrak{r}}\right]$ in~\eqref{eqn::solutions_evenr2}, then the corresponding partition functions still satisfy the system of radial BPZ equations~\eqref{eqn::radialBPZ} with
$\aleph=\frac{(6-\kappa)(\kappa-2)}{8\kappa}-\mathfrak{r}$. In particular, any linear combination of such functions satisfy the same system of radial BPZ equations. We define $\mathcal{W}(\bs{\eta};z)$ to be the event that $\{\left(\partial\Omega_{\bs{\eta}}(z)\cap\partial\Omega\right)\subset\cup_{j=1}^N(x_{2j-1}x_{2j})\}$ and replace $\E_{\alpha}\left[\CR(\Omega\setminus\bs{\eta};z)^{-\mathfrak{r}}\right]$ by $\E_{\alpha}\left[\one\{\mathcal{W}(\bs{\eta};z)\}\CR(\Omega\setminus\bs{\eta};z)^{-\mathfrak{r}}\right]$ in~\eqref{eqn::solutions_evenr2}, then the corresponding partition functions still satisfy the system of radial BPZ equations~\eqref{eqn::radialBPZ} with
$\aleph=\frac{(6-\kappa)(\kappa-2)}{8\kappa}-\mathfrak{r}$. See Figure~\ref{fig::Wetaz}. 

\begin{figure}[ht!]
\begin{subfigure}[b]{0.2\textwidth}
\begin{center}
\includegraphics[width=\textwidth]{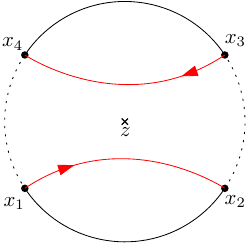}
\end{center}
\caption{}
\end{subfigure}
$\quad$
\begin{subfigure}[b]{0.2\textwidth}
\begin{center}
\includegraphics[width=\textwidth]{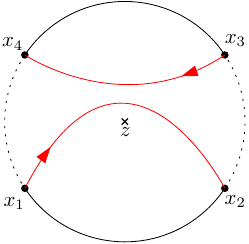}
\end{center}
\caption{}
\end{subfigure}
$\quad$
\begin{subfigure}[b]{0.2\textwidth}
\begin{center}
\includegraphics[width=\textwidth]{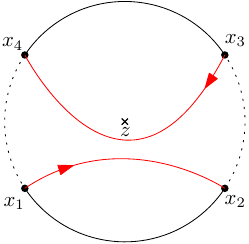}
\end{center}
\caption{}
\end{subfigure}
\\\vspace{0.5cm}
\begin{subfigure}[b]{0.2\textwidth}
\begin{center}
\includegraphics[width=\textwidth]{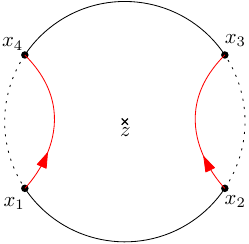}
\end{center}
\caption{}
\end{subfigure}
$\quad$
\begin{subfigure}[b]{0.2\textwidth}
\begin{center}
\includegraphics[width=\textwidth]{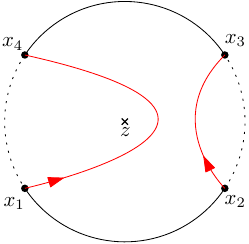}
\end{center}
\caption{}
\end{subfigure}
$\quad$
\begin{subfigure}[b]{0.2\textwidth}
\begin{center}
\includegraphics[width=\textwidth]{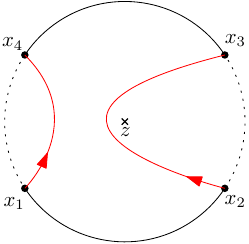}
\end{center}
\caption{}
\end{subfigure}
\caption{\label{fig::Wetaz} 
When $N=2$, there are two possibilities for the link pattern among the four marked points $\{x_1, x_2, x_3, x_4\}$: $\{\{1,2\}, \{3,4\}\}$ and $\{\{1,4\}, \{2,3\}\}$. For each connectivity, the two simple curves $\eta^{(1)}, \eta^{(2)}$ divide the domain into three connected components. If we consider the link pattern together with the location of $z$, there are six possibilities. The event $\mathcal{W}(\bs{\eta};z)$ is the union of the cases in (b), (c) and (d).  
}
\end{figure}

We give the formal definition of the event $\mathcal{W}(\bs{\eta};z)$ for $\kappa\in (0,6]$ in Definition~\ref{def::Wetaz}. When $\kappa\le 4$, it is the same as defined above. When $\kappa\in (4,6]$, as the curves have touchings, the description is more complicated but the idea is the same. 
This is the continuum analogue of the discrete event in Definition~\ref{def::wired_event}. 

\begin{definition}\label{def::Wetaz}
Fix $N\ge 1, \alpha\in\LP_N$ and a polygon with an interior point $(\Omega; x_1, \ldots, x_{2N}; z)$. Fix $\kappa\in (0,6]$ and suppose $(\eta^{(1)}, \ldots, \eta^{(N)})\sim \QQ_{\alpha}$ is global $N$-$\SLE_{\kappa}$ associated to $\alpha$ in $(\Omega; x_1, \ldots, x_{2N})$. We order $\alpha$ as in~\eqref{eqn::linkpattern_order} and orient $\eta^{(j)}$ from $x_{a_j}$ to $x_{b_j}$. We denote by $\mathcal{W}(\bs{\eta}; z)$ the even that $z$ stays to the right of (resp. to the left of) $\eta^{(j)}$ if $a_j$ is odd (resp. if $a_j$ is even) for all $\eta^{(j)}$ such that $\partial\Omega_{\bs{\eta}}(z)\cap\eta^{(j)}\neq\emptyset$.  
For $\mathfrak{r}\in\R$, we define 
	\begin{equation}\label{eqn::solutions_evenr_wired}
		\LZalpharwired(\Omega; x_1, \ldots, x_{2N}; z)=\LZ_{\alpha}(\Omega; x_1, \ldots, x_{2N})\E_{\alpha}\left[\one\{\mathcal{W}(\bs{\eta};z)\}\CR\left(\Omega\setminus\bs{\eta}; z\right)^{-\mathfrak{r}}\right]. 
	\end{equation} 
	Then the expectation in~\eqref{eqn::solutions_evenr_wired} is finite when $\mathfrak{r}<1-\kappa/8$. We denote by 
	\[\QQalpharwired=\QQalpharwired(\Omega; x_1, \ldots, x_{2N}; z)\]
the probability measure of $(\eta^{(1)}, \ldots, \eta^{(N)})\sim\QQ_{\alpha}$ global $N$-$\SLE_{\kappa}$ associated to $\alpha$ in polygon $(\Omega; x_1, \ldots, x_{2N})$ weighted by 
\begin{equation*}
\frac{\LZ_{\alpha}(\Omega; x_1, \ldots, x_{2N})}{\LZalpharwired(\Omega; x_1, \ldots, x_{2N}; z)}\one\{\mathcal{W}(\bs{\eta};z)\}\CR(\Omega\setminus\bs{\eta}; z)^{-\mathfrak{r}}. 
\end{equation*}
\end{definition}

\begin{lemma} \label{lem::LZalpharwired}
Assume the same notations as in Definition~\ref{def::Wetaz}. 
\begin{itemize}
\item For $(\theta_1, \ldots, \theta_{2N})\in\LX_{2N}$, if we write 
	\begin{equation*}
		\LZalpharwired(\theta_1, \ldots, \theta_{2N})=\LZalpharwired(\U; \exp(\ii\theta_1), \ldots, \exp(\ii\theta_{2N});0),
	\end{equation*}
	then $\LZalpharwired: \LX_{2N}\to \R_{>0}$ satisfies the system of radial BPZ equations~\eqref{eqn::radialBPZ} with
	$\aleph=\frac{(6-\kappa)(\kappa-2)}{8\kappa}-\mathfrak{r}$.
\item For $(\theta_1, \ldots, \theta_{2N})\in\LX_{2N}$, suppose $(\eta^{(1)}, \ldots, \eta^{(N)})\sim\QQalpharwired$ in $(\U; \exp(\ii\theta_1), \ldots, \exp(\ii\theta_{2N}); 0)$. Then the law of $\eta^{(1)}$ under $\QQalpharwired$ is the same as radial $\SLE_{\kappa}$ in $(\U; \exp(\ii\theta_1); 0)$ weighted by the following local martingale, up to the first time $\exp(\ii\theta_{2})$ or $\exp(\ii\theta_{2N})$ is disconnected from the origin: 
\begin{equation}\label{eqn::mart_multipleSLE_CR_wired}
M_t(\LZalpharwired)=g_t'(0)^{\mathfrak{r}-\tilde{h}}\times\prod_{j=2}^{2N}\phi_t'(\theta_j)^h\times\LZalpharwired(\xi_t, \phi_t(\theta_2), \ldots, \phi_t(\theta_{2N})). 
\end{equation}  
\end{itemize}	
	\end{lemma}
	\begin{proof}
	This can be proved in the same way as Lemma~\ref{lem::marginal_LZalphar}. 
	\end{proof}

\section{FK-Ising model and proof of Theorem~\ref{thm::FKIsing_cvg}}
\label{sec::FKIsing}

This section is organized as follows. We first give preliminaries on random-cluster models in Section~\ref{subsec::pre_rcm} and give preliminaries on the FK-Ising model in Section~\ref{subsec::con_inv_FK}. Then we complete the proof of Theorem~\ref{thm::FKIsing_cvg} in Section~\ref{subsec::FKIsing_proof}. To simplify the notation, we write $f \lesssim g$ if $f/g$ is bounded by a finite constant from above, and write $f\asymp g$ if $f\lesssim g$ and $g\lesssim f$. For $z\in \mathbb{C}$ and $0<r_1<r_2$, define 
\[B_{r_1}(z):=\{w\in \mathbb{C}: |w-z|<r_1\} \quad\text{and}\quad A_{r_1,r_2}(z):=\{w\in \mathbb{C}: r_1<|w-z|<r_2\}. \]

\subsection{Preliminaries on random-cluster models}
\label{subsec::pre_rcm}
\paragraph*{Random-cluster model.}
Let $\graph = (V(\graph), E(\graph))$ be a finite subgraph of $\Z^2$.
A random-cluster {configuration} 
$\omega=(\omega_e)_{e \in E(\graph)}$ is an element of $\{0,1\}^{E(\graph)}$.
An edge $e \in E(\graph)$ is called {open} (resp.~{closed}) if $\omega_e=1$ (resp.~$\omega_e=0$).
We denote by $o(\omega)$ (resp.~$c(\omega)$) the number of open (resp.~closed) edges in~$\omega$.

We are interested in the connectivity properties of the graph $\omega$ with various boundary conditions. 
The maximal connected\footnote{Two vertices $z$ and $w$ are said to be {connected} by $\omega$ if there exists a sequence  $\{z_j \colon 0\le j\le l\}$  of vertices such that
	$z_0 = z$ and $z_l = w$, and each edge $\edge{z_j}{z_{j+1}}$ is open in $\omega$ for $0 \le j < l$.} components of $\omega$ are called {clusters}.
The boundary conditions encode how the vertices are connected outside of $\graph$.
Precisely, by a {boundary condition} $\bssymb$ we refer to a partition $\bssymb_1 \sqcup \cdots \sqcup \bssymb_m$ of the boundary $\partial \graph$.
Two vertices $z,w \in \partial \graph$ are said to be {wired} in $\bssymb$ if $z,w \in \bssymb_j$ for some common $j$. 
In contrast, {free} boundary segments comprise vertices that are not wired with any other vertex (so the corresponding part $\pi_j$ is a singleton).  
We denote by $\omega^{\bssymb}$
the (quotient) graph obtained from the configuration $\omega$ by identifying the wired vertices in $\bssymb$.

Finally, the {random-cluster model} on $\graph$ with edge-weight $p\in [0,1]$, cluster-weight $q>0$,
and boundary condition $\bssymb$, is the probability measure $\smash{\PRCM^{\bssymb}_{p,q,\graph}}$ on
the set $\{0,1\}^{E(\graph)}$ of configurations $\omega$  defined by
\begin{align*}
	\PRCM^{\bssymb}_{p,q,\graph}[\omega] 
	:= \; & \frac{p^{o(\omega)}(1-p)^{c(\omega)}q^{k(\omega^{\bssymb})}}{\underset{\varpi \in \{0,1\}^{E(\graph)}}{\sum} p^{o(\varpi)}(1-p)^{c(\varpi)}q^{k(\varpi^{\bssymb})} } ,
\end{align*}
where $k(\omega^{\bssymb})$ is the number of 
connected components of the graph $\omega^{\bssymb}$.
For $q=2$, this model is also known as the {FK-Ising model}, while for $q=1$, it is simply the Bernoulli bond percolation (assigning independent values for each $\omega_e$). For $A,B\subseteq G$, we write $\{A\leftrightarrow B\}$ for the event that there exists an open path connecting $A$ to $B$; if $A=\{v\}$ for some vertex $v$, then we write $\{v\leftrightarrow B\}$ for the event $\{\{v\}\leftrightarrow B\}$.

In the present article, we focus on the random-cluster model on finite subgraphs of the square lattice $\mathbb{Z}^2$, or the scaled square lattice $\delta\mathbb{Z}^2$.
It has been proven for the range $q \in [1,4]$ in~\cite{DuminilSidoraviciusTassionContinuityPhaseTransition} 
that when the edge-weight is chosen suitably, namely as (the critical, self-dual value)
\begin{align*}
	p = p_c(q) := \frac{\sqrt{q}}{1+\sqrt{q}} ,
\end{align*}
then the random-cluster model exhibits a {continuous phase transition}.

\paragraph*{RSW estimates.}

For $1\leq q< 4$ and $p=p_c(q)$, we have the following strong RSW estimates.  For a discrete quad $(G;a,b,c,d)$, we denote by $L=L(G; a, b, c, d)$ the discrete extremal distance between $(ab)$ and $(cd)$ in $G$; see~\cite[Section~6]{ChelkakRobustComplexAnalysis}. The discrete extremal distance is uniformly comparable to its continuous counterpart, i.e., the classical extremal distance.

\begin{lemma}\textnormal{\cite[Theorem~1.2]{DCMTRCMFractalProperties}} \label{lem::RSW}
	Let $q\in [1,4)$. 	For each $L_0>0$, there exists $c(L_0,q)>0$ such that the following holds: for any discrete quad $(G;a,b,c,d)$ with $L(G; a, b, c, d)\le L_0$ and any boundary condition $\pi$, we have
	\begin{equation*}
		\mu_{p_c(q), q,G}^{\pi} \left[(ab)\leftrightarrow(cd)\right]\geq c(L_0,q).
	\end{equation*}
\end{lemma}

\paragraph*{Discrete polygons.} A {discrete (topological) polygon}
is a finite simply connected subgraph 
of $\Z^2$, or $\delta \Z^2$, 
with $2N$ marked boundary points in counterclockwise order. We now give its precise definition. 

\begin{enumerate}[leftmargin=*]
	\item  
	First, we define the {medial polygon}. Edges of the medial lattice $(\Z^2)^\diamond$ are oriented as follows: edges of each face containing a vertex of $\Z^2$ are oriented clockwise, and edges of each face containing a vertex of $(\Z^2)^{\bullet}$ are oriented counterclockwise. 
	Let $x_1^\diamond,\ldots, x_{2N}^\diamond$ be $2N$ distinct medial vertices. Let $(x_1^\diamond \, x_2^\diamond), (x_2^\diamond \, x_3^\diamond), \ldots , (x_{2N}^\diamond  \, x_{1}^\diamond)$ be $2N$ oriented paths on $(\Z^2)^\diamond$ satisfying the following conditions\footnote{Throughout, we use the convention that $x_{2N+1}^\diamond := x_{1}^\diamond$.}: 
	\begin{itemize}[leftmargin=1.0em]
		\item 
		the path $(x_{2i-1}^\diamond \, x_{2i}^\diamond)$ consists of counterclockwise oriented edges for $1\leq i \leq N$; 
		
		\item  
		the path $(x_{2i}^\diamond \, x_{2i+1}^\diamond)$ consists of clockwise oriented edges for $1\leq i \leq N$; 
		
		\item  
		all paths are edge-self-avoiding and satisfy $(x_{i-1}^\diamond \, x_i^\diamond) \cap (x_i^\diamond \, x_{i+1}^\diamond) = \{x_i^\diamond\}$ for $1\leq i \leq 2N$;

		\item  
		if $j\notin \{i+1,i-1\}$, then $(x_{i-1}^\diamond \, x_{i}^\diamond) \cap (x_{j-1}^\diamond \, x_j^\diamond) = \emptyset$; 
		
		\item  
		the infinite connected component of 
		$(\Z^2)^\diamond\setminus \smash{\bigcup_{i=1}^{2N}} (x_i^\diamond \, x_{i+1}^\diamond)$ 
		is on the right of the oriented path~$(x_1^\diamond \, x_2^\diamond)$. 
	\end{itemize}
	Given $\{(x_i^\diamond \, x_{i+1}^\diamond) \colon 1\leq i\leq 2N\}$, the medial polygon $(\Omega^\diamond; x_1^\diamond,\ldots, x_{2N}^\diamond)$ is defined 
	as the subgraph of $(\Z^2)^\diamond$ induced by the vertices lying on or enclosed by the non-oriented loop obtained by concatenating all of $(x_i^\diamond \, x_{i+1}^\diamond)$. 
	For each $i \in \{1,2,\ldots,2N\}$, the {outer corner}  $y_{i}^{\diamond}\in (\mathbb{Z}^2)^\diamond\setminus\Omega^\diamond$ is defined to be a medial vertex adjacent to $x_i^\diamond$, and the {outer corner edge} $e_i^\diamond$ is defined to be the medial edge connecting them.
	
	\item  
	Second, we define the {primal polygon} 
	$(\Omega;x_1,\ldots,x_{2N})$ induced by $(\Omega^\diamond;x_1^\diamond,\ldots,x_{2N}^\diamond)$ as follows: 
	\begin{itemize}[leftmargin=1.0em]
		\item $\Omega$ is a subgraph of $\mathbb{Z}^2$;
		\item 
		its edge set $E(\Omega)$ consists of edges passing through endpoints of medial edges in 
		$E(\Omega^\diamond)\setminus \smash{\bigcup_{i=1}^N} (x_{2i}^\diamond \, x_{2i+1}^\diamond)$; 
		
		\item  
		its vertex set $V(\Omega)$ consists of endpoints of edges in $E(\Omega)$; 
		
		\item  
		the marked boundary vertex $x_i$ is defined to be the vertex in $\Omega$ nearest to $x_i^\diamond$ for each $1\leq i\leq 2N$; 
		
		\item  
		the arc $(x_{2i-1} \, x_{2i})$ is the set of edges whose midpoints are vertices in $(x_{2i-1}^\diamond \, x_{2i}^\diamond)\cap \partial \Omega^\diamond$ for $1\leq i\leq N$.
	\end{itemize}

	\item  
	Third, we define the {dual polygon} $(\Omega^{\bullet};x_1^{\bullet},\ldots,x_{2N}^{\bullet})$ induced by $(\Omega^\diamond; x_1^\diamond,\ldots,x_{2N}^\diamond)$ in a similar way: $\Omega^{\bullet}$ is the subgraph of $(\Z^2)^{\bullet}$ with 
	\begin{itemize}[leftmargin=1.0em]
		\item 
		edge set consisting of edges passing through endpoints of medial edges in $E(\Omega^\diamond)\setminus \smash{\bigcup_{i=1}^{N}} (x_{2i-1}^\diamond \, x_{2i}^\diamond)$; 
		\item and vertex set consisting of the endpoints of these edges. 
	\end{itemize}
	For each $i \in \{1,2,\ldots,2N\}$, 
	the marked boundary vertex $x_i^{\bullet}$ is defined to be the vertex in $\Omega^{\bullet}$ nearest to $x_i^\diamond$; and 
	for each $i \in \{1,2,\ldots,N\}$, 
	the boundary arc $(x_{2i}^{\bullet} \, x_{2i+1}^{\bullet})$ is defined to be 
	the set of edges whose midpoints are vertices in $(x_{2i}^\diamond \, x_{2i+1}^\diamond)\cap \Omega^\diamond$. 
\end{enumerate}

\paragraph*{Admissible domains} We say a simply connected subgraph $\Omega^{\delta,\diamond}$ of $(\delta\mathbb{Z})^{\diamond}$ is an {admissible medial domain} if its boundary  consists of counterclockwise oriented edges. Suppose that $\Omega^{\delta,\diamond}$ is an admissible domain. Then we denote the law of the critical FK-Ising model on the primal domain $\Omega^{\delta}$ with the wired boundary condition by $\mathbb{P}_{\Omega,\mathrm{w}}^{\delta}$.

\paragraph*{Boundary conditions.}
In this work, 
we shall focus on the critical FK-Ising model on the primal polygon $(\Omega;x_1,\ldots,x_{2N}) = (\Omega^\delta; x_1^\delta,\ldots,x_{2N}^\delta)$, with the alternating boundary condition~\eqref{eqn::bc}:
\begin{align*}
	(x_{2j-1}^{\delta} \, x_{2j}^{\delta}) \textnormal{ is wired,} \qquad \textnormal{ for all } j \in\{1,2,\ldots, N\} ,
\end{align*}
and these $N$ wired arcs are not wired outside of $\Omega^{\delta}$. This boundary condition is encoded in the unnested link pattern: 
\[\unnested=\{\{1,2\},\{3,4\},\ldots ,\{2N-1,2N\}\}.\]                                                    We denote by $\PP^{\delta}_{\Omega}$ the law, and by $\mathbb{E}_{\Omega}^{\delta}$ the expectation, of the critical model on $(\Omega^{\delta}; x_{1}^{\delta},\ldots,x_{2N}^{\delta})$ with the boundary condition described above, where the cluster-weight has the fixed value $q=2$ in this section. 

\paragraph*{Loop representation and interfaces.}
Let $\omega \in \{0,1\}^{E(\Omega^\delta)}$ be a configuration 
with the alternating boundary condition~\eqref{eqn::bc} on the primal polygon $(\Omega^\delta; x_1^\delta,\ldots,x_{2N}^\delta)$. 
The dual configuration $\omega^{\bullet}$ on $\Omega^{\bullet}$ induced by $\omega$  is defined by $\omega^{\bullet}_e = 1 - \omega_e$. We say that an edge $e \in E(\Omega^{\bullet})$ is {dual-open} (resp.~{dual-closed}) if $\omega^{\bullet}_e=1$ (resp.~$\omega^{\bullet}_e=0$).
Given $\omega$, we can draw self-avoiding 
paths on the medial graph $\Omega^{\delta, \diamond}$ between $\omega$ and $\omega^{\bullet}$ as follows: 
a path arriving at a vertex of $\Omega^{\delta,\diamond }$ always makes a turn of $\pm\pi/2$, so as not to cross the open or dual-open edges through this vertex. 
The {loop representation} of $\omega$ consists of a number of loops and $N$ pairwise-disjoint and self-avoiding {interfaces} connecting the $2N$ outer corners $y_{1}^{\delta,\diamond}, \ldots,y_{2N}^{\delta,\diamond}$ of  the medial polygon $(\Omega^{\delta,\diamond};x_1^{\delta,\diamond},\ldots,x_{2N}^{\delta,\diamond})$. For each $i\in \{1,2,\ldots,2N\}$, we shall denote by $\eta_i^\delta$ the interface starting from the medial vertex $y_{i}^{\delta,\diamond}$
(and we also refer to it as the interface starting from the boundary point $x_{i}^{\delta,\diamond}$). We denote by $\vartheta^{\delta}$ the (random) link pattern of multiple interfaces $(\eta_1^{\delta},\ldots,\eta_{2N}^{\delta})$, which takes value in $\LP_N$. 
                                                                                                                                                                                                                                                          
Note that for the model on an admissible domain $\Omega^{\delta}$ with the wired boundary condition, we can also define its loop representation as above, which consists of interface loops only.

\paragraph*{Scaling limits}
We  need a topology for the interfaces, which we regard as (images of) continuous mappings from $[0,1]$ to $\C$ modulo reparameterization, i.e., planar oriented curves. 
For a simply connected domain $\Omega \subsetneq \C$, we will consider curves in $\overline{\Omega}$.
For definiteness, we map $\Omega$ onto the unit disc 
$\U := \{ z \in \C \colon |z| < 1 \}$: 
for this we shall fix\footnote{The metric~\eqref{eq::curve_metric} depends on the choice of the conformal map $\Phi$, but the induced topology does not.} 
any conformal map $\Phi$ from $\Omega$ onto $\U$. 
Then, we endow the curves with the metric
\begin{align} \label{eq::curve_metric} 
	\metric(\eta_1, \eta_2) 
	:= \inf_{\psi_1, \psi_2} \sup_{t\in [0,1]} |\Phi(\eta_1(\psi_1(t)))-\Phi(\eta_2(\psi_2(t)))| ,
\end{align}
where the infimum is taken over all increasing homeomorphisms $\psi_1, \psi_2 \colon [0,1]\to[0,1]$.
The space of continuous curves on $\overline{\Omega}$ modulo reparameterizations then becomes a complete separable metric space. Let $k\geq 1$, for two collections of curves $(\eta_1,\ldots,\eta_{k})$ and $(\gamma_1,\ldots,\gamma_{k})$, we define 
\begin{equation} \label{eqn::metric_multiple}
	\dist\left((\eta_1,\ldots,\eta_{k}),(\gamma_1,\ldots,\gamma_{k})\right):=\min_{1\leq j\leq k} \dist(\eta_j,\gamma_j). 
\end{equation}
We denote by $X_k$ the space of the collections of curves $(\eta_1,\ldots,\eta_{k})$ endowed with the metric~\eqref{eqn::metric_multiple}. 

We also need a topology for the collection of loops in the loop representation. An oriented continuous curve $\gamma:[0,1]\to \mathbb{C}$ with $\gamma(0)=\gamma(1)$ is called a loop. Then, we define a distance between two closed sets of loops, $\Gamma_1$ and $\Gamma_2$, as follows:
\begin{align} \label{eqn::LE_metric}
	\Dist(\Gamma_1,\Gamma_2):=\inf\{\epsilon>0: \forall \gamma_1\in \Gamma_1, \text{ }\exists \gamma_2\in \Gamma_2 \text{ s.t. } \metric(\gamma_1,\gamma_2)\leq \epsilon \text{ and vice versa}\}. 
\end{align}
The space of collections of loops with distance $\Dist$ is also complete and separable.

\paragraph*{Convergence of polygons.}
To investigate the scaling limit, we use three kinds of convergence of domains: convergence of domains in the Carath\'edory sense~\cite{Pommerenke} (used in Lemma~\ref{lem::limit_one_point} and Proposition~\ref{prop::CLE}), convergence of polygons in the close-Carath\'eodory sense~\cite{KarrilaConformalImage, KarrilaMultipleSLELocalGlobal} (used in Propositions~\ref{prop::cvg_global_interfaces} and~\ref{prop::cvg_proba}), and the convergence of polygons with an interior point in the close-Carath\'edory sense (used in Conjecture~\ref{conj::cvg_interfaces}).
Abusing notation, for a discrete polygon, we will occasionally denote by $\Omega^{\delta}$ also the open simply connected subset of $\C$ defined as the interior of the set $\overline{\Omega}^{\delta}$ comprising all vertices, edges, and faces of the polygon $\Omega^{\delta}$. A sequence of domains
$\Omega^{\delta}$ converges to $\Omega$ in the Carath\'{e}odory sense as $\delta\to0$ if there exist conformal maps $\varphi_{\delta}$ from $\Omega^{\delta}$ onto
$\U$, and a conformal map $\varphi$ from $\Omega$ onto $\U$, such that $\varphi_{\delta}^{-1}\to\varphi^{-1}$ locally uniformly on $\U$ as $\delta\to 0$.

\begin{definition} \label{def:closeCara}
	We say that a sequence of discrete polygons $(\Omega^{\delta}; x_1^{\delta}, \ldots, x_{2N}^{\delta})$  
	converges as $\delta \to 0$ to a polygon $(\Omega; x_1, \ldots, x_{2N})$ in the {close-Carath\'{e}odory sense} if 
	\begin{enumerate}
		\item  $x_{j}^{\delta}\to x_j$ for all $1\leq j\leq 2N$; and
		\item there exist conformal maps $\varphi_{\delta}$ from $\Omega^{\delta}$ onto $\mathbb{U}$ and a conformal map $\varphi$ from $\Omega$ onto $\mathbb{U}$, such that $\varphi_{\delta}^{-1}\to \varphi^{-1}$ locally uniformly on $\mathbb{U}$, and moreover $\varphi_{\delta}(x_j^{\delta})\to \varphi(x_j)$ for all $1\leq j\leq 2N$; and 
		 \item for a given reference point $u\in \Omega$ and small enough $r>0$, let $S_r$ be the arc of $\partial B(x_j,r)\cap\Omega$ disconnecting \textnormal{(}in $\Omega$\textnormal{)} $x_j$ from $u$ and from all other arcs of this set, then for small enough $r$ and $\delta$ (depending on $r$), the boundary point $x_j^{\delta}$ is connected to the midpoint of $S_r$ inside $\Omega^{\delta}\cap B(x_j,r)$.
	\end{enumerate}

If we also have $z^{\delta}\to z$ as $\delta\to 0$ for some $z^{\delta}\in \Omega^{\delta}$ and $z\in\Omega$, then we say that $(\Omega^{\delta}; x_1^{\delta}, \ldots, x_{2N}^{\delta};z^{\delta})$ converges as $\delta\to0$ to $(\Omega;x_1,\ldots,x_{2N};z)$ in the close-Carath\'edory sense.   
\end{definition}

\begin{lemma}\label{lem::FKIsing_tightness}
	Assume the same notations as in 
Conjecture~\ref{conj::cvg_interfaces}. 
	Fix $i\in\{1,2, \ldots, 2N\}$. The family of laws of $\{(\eta_1^{\delta},\ldots\eta_{2N}^{\delta})\}_{\delta>0}$ is  
	precompact in the space of curves with metric~\eqref{eqn::metric_multiple}.
	Furthermore, each $\eta_j$ in any subsequential limit $(\eta_1,\ldots,\eta_{2N})$ does not hit any other point in $\{x_1, x_2, \ldots, x_{2N}\}$ than its two endpoints, almost surely.
\end{lemma}

\begin{proof}
	Without conditioning on the one-arm event~\eqref{eqn::onearm_event_def}, the proof is standard nowadays. 
	For instance, the case where $q=2$ is treated in~\cite[Lemmas~4.1 and~5.4]{IzyurovMultipleFKIsing}. 
	The main tools are RSW bounds from~\cite{DuminilCopinHonglerNolinRSWFKIsing, KemppainenSmirnovRandomCurves} 
	--- see also~\cite{KarrilaConformalImage, KarrilaMultipleSLELocalGlobal}. 
The case of general $q\in [1,4)$ follows from~\cite[Theorem~6]{DuminilSidoraviciusTassionContinuityPhaseTransition} and~\cite[Section~1.4]{DCMTRCMFractalProperties}. 
The argument  still works for the interfaces conditional on the one-arm event~\eqref{eqn::onearm_event_def} due to the facts that~\eqref{eqn::onearm_event_def} is an increasing event and that we have the FKG inequality. 
\end{proof}

\subsection{Preliminaries on FK-Ising model}\label{subsec::con_inv_FK}
We collect two results concerning the conformal invariance of FK-Ising multiple interfaces (Proposition~\ref{prop::cvg_global_interfaces}) and their connection probabilities (Proposition~\ref{prop::cvg_proba}). 

\begin{proposition}\textnormal{\cite[Proposition~1.4]{BeffaraPeltolaWuUniqueness}} \label{prop::cvg_global_interfaces}
Fix a polygon $(\Omega; x_1, \ldots, x_{2N})$ and suppose a sequence of medial domains $(\Omega^{\delta, \diamond}; x_1^{\delta, \diamond}, \ldots, x_{2N}^{\delta, \diamond})$ converges to $(\Omega; x_1, \ldots, x_{2N})$ in the close-Carath\'{e}odory sense. Consider the critical FK-Ising model on the primal domain $(\Omega^{\delta}; x_1^{\delta}, \ldots, x_{2N}^{\delta})$ with alternating boundary condition~\eqref{eqn::bc}.  Fix $\alpha\in\LP_N$. Then the law of the collection of multiple interfaces conditional on the event $\{\vartheta^{\delta}=\alpha\}$ converges weakly under the topology induced by $\dist$ in~\eqref{eqn::metric_multiple} to global $N$-$\mathrm{SLE}_{16/3}$ associated to $\alpha$ in $(\Omega;x_1,\ldots,x_{2N})$.
\end{proposition}

\begin{definition} \label{def::meander}
	A {meander} formed from two link patterns $\alpha,\beta\in\LP_N$ is the planar diagram obtained by placing $\alpha$ and the horizontal reflection $\beta$ on top of each other. An example of a meander is
	\begin{align*} 
		\alpha \quad = \quad \vcenter{\hbox{\includegraphics[scale=0.275]{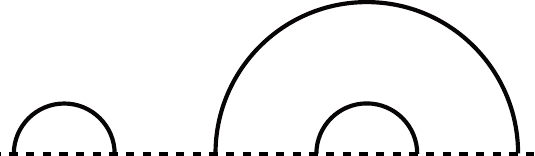}}} 
		\quad  , \quad  
		\beta \quad = \quad\vcenter{\hbox{\includegraphics[scale=0.275]{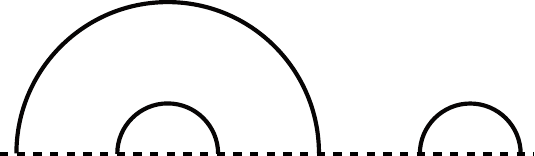}}} 
		\quad\quad\quad \Longrightarrow \quad\quad\quad
		\vcenter{\hbox{\includegraphics[scale=0.275]{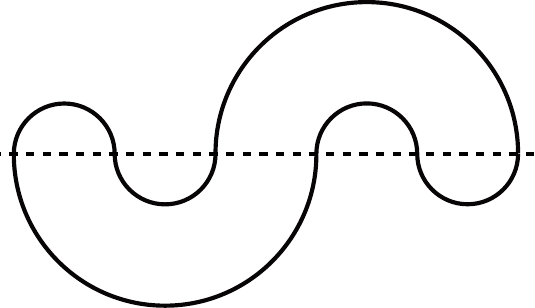}}} .
	\end{align*}
	We denote by $\LL_{\alpha,\beta}$ the number of loops in the meander formed from $\alpha$ and $\beta$. Fix $q\in (0,4)$. We define the {meander matrix} $\{\LM_{\alpha, \beta}(q) \colon \alpha,\beta\in\LP_N\}$ via 
	\begin{align*}
		\LM_{\alpha,\beta}(q) := \sqrt{q}^{\; \LL_{\alpha,\beta}}. 
	\end{align*}
\end{definition}

\begin{proposition}\textnormal{~\cite[Theorem~1.8]{FengPeltolaWuConnectionProbaFKIsing}} \label{prop::cvg_proba}
	Assume the same notations as in Proposition~\ref{prop::cvg_global_interfaces}.
	Recall that $\vartheta^{\delta}$ is the link pattern given by multiple interfaces $(\eta_1^{\delta},\ldots,\eta_{2N}^{\delta})$. Then we have 
	\begin{align*}
	\lim_{\delta\to 0}\mathbb{P}_{\Omega}^{\delta}[\vartheta^{\delta}=\alpha]=\LM_{\unnested,\alpha}(2)\frac{\mathcal{Z}_{\alpha}(\theta_1,\ldots,\theta_{2N})}{\mathcal{F}(\theta_1,\ldots,\theta_{2N})}, \quad \text{for all }\alpha\in \LP_N,
 	\end{align*}

	where the function $\LF$ is defined by
	\begin{align*}
		\mathcal{F}(\theta_1,\ldots,\theta_{2N})= \; & \prod_{s=1}^N \big(\sin((\theta_{2s}-\theta_{2s-1})/2)\big)^{-1/8}
		\bigg(\sum_{\boldsymbol{\mu} \in \{\pm 1\}^N}\prod_{1\le s < t \le N} \chi(\theta_{2s-1},\theta_{2t-1},\theta_{2t},\theta_{2s})^{ \mu_s \mu_t/4}\bigg)^{1/2},
	\end{align*}
	where 
	\begin{align*}
		\chi(\hat{\theta}_1,\hat{\theta}_2,\hat{\theta}_3,\hat{\theta}_4):= \left| \frac{\sin((\hat{\theta}_2-\hat{\theta}_1)/2)\sin((\hat{\theta}_4-\hat{\theta}_3)/2)}{\sin((\hat{\theta}_3-\hat{\theta}_1)/2)\sin((\hat{\theta}_4-\hat{\theta}_2)/2)}\right|.
	\end{align*}
\end{proposition}

We now consider the crtitical FK-Ising model in a polygon $(\Omega^{\delta};x_1^{\delta},\ldots,x_{2N}^{\delta})$ with the alternating boundary condition~\eqref{eqn::bc}, and let $\eta_j^{\delta}$ be the interface starting from $x_j^{\delta,\diamond}$, $1\leq j\leq 2N$. Let $z^{\delta}\in \Omega^{\delta}$ and write $\bs{\eta}^{\delta}=\cup_{j=1}^{2N}\eta_j^{\delta}$. Note that the event $\{z^{\delta}\leftrightarrow \cup_{j=1}^N (x_{2j-1}^{\delta} x_{2j}^{\delta})\}$  would impose some topological restrictions on the locations of $\bs{\eta}^{\delta}$ and $z^{\delta}$. We now elaborate on these restrictions. 
\begin{definition} \label{def::wired_event}
	Let $\Omega_{\bs{\eta}^{\delta}}^{\delta,\diamond}(z^{\delta})$ be the connected component of  $\Omega^{\delta,\diamond}\setminus \bs{\eta}^{\delta}$ containing $z^{\delta}$. Let $\Omega_{\bs{\eta}^{\delta}}^{\delta}(z^{\delta})$ be the connected component of $\Omega^{\delta}\setminus \bs{\eta}^{\delta}$ containing $z^{\delta}$. We give each $\eta_j^{\delta}$ an orientation such that it always has open edges on its right and  dual-open edges on its left. We also give $\cup_{j=1}^N (x_{2j-1}^{\delta,\diamond}x_{2j}^{\delta,\diamond})$ (resp., $\cup_{j=1}^N(x_{2j}^{\delta,\diamond}x_{2j+1}^{\delta,\diamond})$) an orientation such that it has $\Omega^{\delta,\diamond}$ on its right (resp., left). We then denote by $\mathcal{W}(\bs{\eta}^\delta;z^{\delta})$ the event that the boundary of $\Omega_{\bs{\eta}^{\delta}}^{\delta,\diamond}(z^{\delta})$ is oriented clockwise. 
\end{definition}
Using the domain Markov property of the FK-Ising model, we have 
\begin{align} \label{eqn::discrete_one_arm_decom}
	\mathbb{P}_{\Omega}^{\delta}\left[z^{\delta}\leftrightarrow \cup_{j=1}^N(x_{2j-1}^{\delta}x_{2j}^{\delta})\right]=\mathbb{E}_{\Omega}^{\delta}\left[\one \left\{\LW(\bs{\eta}^{\delta};z^{\delta})  \right\} \mathbb{P}_{\Omega_{\bs{\eta}^{\delta}}^{\delta}(z^{\delta}),\mathrm{w}}^{\delta}\left[z^{\delta}\leftrightarrow\partial \Omega_{\bs{\eta}^{\delta}}^{\delta}(z^{\delta})\right]\right].
\end{align}

\begin{corollary} \label{coro::limit_one_point}
Assume the same notations as in Conjecture~\ref{conj::cvg_interfaces} and fix $q=2$.  Then we have 
\begin{align}\label{eqn::coro::limit_one_point}
	\lim_{\delta\to 0}\frac{\mathbb{P}_{\Omega}^{\delta}\left[z^{\delta}\leftrightarrow \cup_{j=1}^N(x_{2j-1}^{\delta}x_{2j}^{\delta})\right]}{\mathbb{P}^{\delta}_{\mathbb{U},\mathrm{w}}[0\leftrightarrow \partial \mathbb{U}^{\delta}]}=\sum_{\alpha\in \LP_N} \LM_{\unnested,\alpha}(2)\frac{\mathcal{Z}_{\alpha}(\theta_1,\ldots,\theta_{2N})}{\mathcal{F}(\theta_1,\ldots,\theta_{2N})}\mathbb{E}_{\alpha}\left[\one \left\{\LW(\bs{\eta};z)\right\}\CR(\Omega\setminus \bs{\eta};z)^{-1/8} \right],
\end{align}
where $\mathbb{E}_{\alpha}$ is the expectation with respect to the law of global $N$-$\mathrm{SLE}_{16/3}$ associated to $\alpha$ in the polygon $(\Omega;x_1,\ldots,x_{2N})$, and the event $\mathcal{W}(\bs{\eta};z)$ is defined in Definition~\ref{def::Wetaz}.
\end{corollary}
\begin{proof}
	It follows from the observation~\eqref{eqn::discrete_one_arm_decom} that 
	\begin{align*}
	\mathbb{P}_{\Omega}^{\delta}\left[z^{\delta}\leftrightarrow \cup_{j=1}^N(x_{2j-1}^{\delta}x_{2j}^{\delta})\right]=\sum_{\alpha\in\LP_N}\mathbb{E}_{\alpha}^{\delta}\left[\one \left\{\LW(\bs{\eta}^{\delta};z^{\delta})  \right\} \mathbb{P}_{\Omega_{\bs{\eta}^{\delta}}^{\delta}(z^{\delta}),\mathrm{w}}^{\delta}\left[z^{\delta}\leftrightarrow\partial \Omega_{\bs{\eta}^{\delta}}^{\delta}(z^{\delta})\right]\right]\times \mathbb{P}_{\Omega}^{\delta}\left[\vartheta^{\delta}=\alpha\right],
	\end{align*}
	where $\mathbb{E}_{\alpha}^{\delta}$ is the expectation with respect to the law of $\mathbb{P}_{\Omega}^{\delta}$ conditional on the event $\{\vartheta^{\delta}=\alpha\}$.
Thanks to Proposition~\ref{prop::cvg_proba}, it suffices to show that 
\begin{align} \label{eqn::coro::limit_one_point_aux1}
	\lim_{\delta\to 0}\frac{\mathbb{E}_{\alpha}^{\delta}\left[\one \left\{\LW(\bs{\eta}^{\delta};z^{\delta})  \right\} \mathbb{P}_{\Omega_{\bs{\eta}^{\delta}}^{\delta}(z^{\delta}),\mathrm{w}}^{\delta}\left[z^{\delta}\leftrightarrow\partial \Omega_{\bs{\eta}^{\delta}}^{\delta}(z^{\delta})\right]\right]}{\mathbb{P}_{\mathbb{U},\mathrm{w}}^{\delta}\left[0\leftrightarrow \partial \mathbb{U}^{\delta}\right]}= \mathbb{E}_{\alpha}\left[\one \left\{\LW(\bs{\eta};z)\right\}\CR(\Omega\setminus \bs{\eta};z)^{-1/8} \right],\enspace \text{for all }\alpha\in\LP_N.
\end{align}

Fix $\alpha\in \LP_N$ and write $\alpha=\{\{a_1,b_1\},\ldots, \{a_N,b_N\}\}$. Conditional on the event $\{\vartheta^{\delta}=\alpha\}$, for $1\leq j\leq N$, let $\eta^{(j,\delta)}$ be the curve in $(\eta_1^{\delta},\ldots,\eta_{2N}^{\delta})$ having $x_{a_j^{\delta}}^{\delta,\diamond}$ and $x_{b_j^{\delta}}^{\delta,\diamond}$ as endpoints. According to Proposition~\ref{prop::cvg_global_interfaces} (also by coupling them into the same probability space), we may assume that $(\eta^{(1,\delta)},\ldots,\eta^{(N,\delta)})$ converges almost surely as $\delta\to0$ to $(\eta^{(1)},\ldots,\eta^{(N)})\sim \mathbb{P}_{\alpha}$ under the metric~\eqref{eqn::metric_multiple}. Then the discrete domains $\Omega_{\bs{\eta}^{\delta}}^{\delta}(z^{\delta})$ converge almost surely as $\delta\to 0$ to $\Omega_{\bs{\eta}}(z)$ in the Carath\'edory sense.
Let $0<\epsilon\ll 1$. On the one hand, a standard application of the FKG inequality and RSW estimates implies that 
\[\frac{\one \left\{\LW(\bs{\eta}^{\delta};z^{\delta}),\; \bs{\eta}^{\delta}\cap B_{\epsilon}(z^{\delta})=\emptyset  \right\} \mathbb{P}_{\Omega_{\bs{\eta}^{\delta}}^{\delta}(z^{\delta}),\mathrm{w}}^{\delta}\left[z^{\delta}\leftrightarrow\partial \Omega_{\bs{\eta}^{\delta}}^{\delta}(z^{\delta})\right]}{\mathbb{P}_{\mathbb{U},\mathrm{w}}^{\delta}\left[0\leftrightarrow \partial \mathbb{U}^{\delta}\right]}\lesssim 1,\quad \text{as }\delta\to 0.\]
It then follows from Lemma~\ref{lem::limit_one_point} and the dominated convergence theorem that 
\begin{align}\label{eqn::coro::limit_one_point_aux2}
\begin{split}
			&\lim_{\delta\to 0}\frac{\mathbb{E}_{\alpha}^{\delta}\left[\one \left\{\LW(\bs{\eta}^{\delta};z^{\delta}),\; \bs{\eta}^{\delta}\cap B_{\epsilon}(z^{\delta})=\emptyset  \right\} \mathbb{P}_{\Omega_{\bs{\eta}^{\delta}}^{\delta}(z^{\delta}),\mathrm{w}}^{\delta}\left[z^{\delta}\leftrightarrow\partial \Omega_{\bs{\eta}^{\delta}}^{\delta}(z^{\delta})\right]\right]}{\mathbb{P}_{\mathbb{U},\mathrm{w}}^{\delta}\left[0\leftrightarrow \partial \mathbb{U}^{\delta}\right]}\\
	&	\qquad= \mathbb{E}_{\alpha}\left[\one \left\{\LW(\bs{\eta};z),\; \bs{\eta}\cap B_{\epsilon}(z)=\emptyset\right\}\CR(\Omega\setminus \bs{\eta};z)^{-1/8} \right].
\end{split}
\end{align}
On the other hand, note that on the event $\{\bs{\eta}^{\delta}\cap B_{\epsilon}(z^{\delta})\neq \emptyset\}$, there exist one open path and one dual-open path connecting $\partial \Omega^{\delta,\diamond}$ to $\partial B_{\epsilon}(z^{\delta})$. Then a standard application of the FKG inequality and strong RSW estimates in Lemma~\ref{lem::RSW} (see e.g.,~\cite[Proof of Corollary~6.7]{DCMTRCMFractalProperties}) implies that there exists a constant $c_1$ which is independent of $\delta$ and $\epsilon$ such that
\begin{align}\label{eqn::coro::limit_one_point_aux3}                                                                                                                                                                                                                                                                                                                                                                                    
	\frac{\mathbb{E}_{\alpha}^{\delta}\left[\one \left\{\LW(\bs{\eta}^{\delta};z^{\delta}),\; \bs{\eta}^{\delta}\cap B_{\epsilon}(z^{\delta})\neq\emptyset  \right\} \mathbb{P}_{\Omega_{\bs{\eta}^{\delta}}^{\delta}(z^{\delta}),\mathrm{w}}^{\delta}\left[z^{\delta}\leftrightarrow\partial \Omega_{\bs{\eta}^{\delta}}^{\delta}(z^{\delta})\right]\right]}{\mathbb{P}_{\mathbb{U},\mathrm{w}}^{\delta}\left[0\leftrightarrow \partial \mathbb{U}^{\delta}\right]}\lesssim \epsilon^{c_1}.
\end{align}                                                                                                                                                                                                                                                                                                                                                                                                                                                                            

Combining~\eqref{eqn::coro::limit_one_point_aux2} with~\eqref{eqn::coro::limit_one_point_aux3} and letting $\epsilon\to 0$ give~\eqref{eqn::coro::limit_one_point_aux1}, as desired.
\end{proof}

\subsection{Proof of Theorem~\ref{thm::FKIsing_cvg}}
\label{subsec::FKIsing_proof}

\begin{definition} \label{def::partition_random_cluster}
Fix $\kappa\in (4,6]$.	The cluster-weight $q$ and parameter $\kappa$ are related through~\eqref{eqn::qkappa}. We define 
	\begin{align} \label{eqn::def_partition_random_cluster}
		\LG^{(\mathfrak{r})}(\theta_1,\ldots,\theta_{2N})=\sum_{\alpha\in\LP_N} \LM_{\unnested,\alpha}(q) \LZalpharwired(\theta_1,\ldots,\theta_{2N}),\quad \text{for }(\theta_1,\ldots,\theta_{2N})\in\chamber_{2N},
	\end{align}
	where $\LZalpharwired$ is defined in Definition~\ref{def::Wetaz}. 
\end{definition}
From 
Lemma~\ref{lem::LZalpharwired}, the function $\LG^{(\mathfrak{r})}$ defined in~\eqref{eqn::def_partition_random_cluster} satisfies the system of radial BPZ equations~\eqref{eqn::radialBPZ} with
$\aleph=\frac{(6-\kappa)(\kappa-2)}{8\kappa}-\mathfrak{r}>-\frac{3}{2\kappa}$.

\begin{proof}[Proof of Theorem~\ref{thm::FKIsing_cvg}]
Fix $\kappa=16/3$. Without loss of generality, we may assume $\Omega=\mathbb{U}$, $z=0$ and $\varphi$ is the identity map. 
	
	By Lemma~\ref{lem::FKIsing_tightness}, we may choose a subsequence $\delta_n\to 0$ such that $(\eta_1^{\delta_n},\ldots,\eta_{2N}^{\delta_n})$ converges weakly in the metric~\eqref{eqn::metric_multiple} as $n\to\infty$. We denote the limit by $(\eta_1,\ldots,\eta_{2N})$. 
	Let $\vartheta=\{\{a_1,b_1\},\ldots,\{a_N,b_N\}\}\in \LP_N$ be the link pattern given by $(\eta_1,\ldots,\eta_{2N})$. For $1\leq j\leq N$, let $\eta^{(j)}$ be the curve in $(\eta_1,\ldots,\eta_{2N})$ having $x_{a_j}$ and $x_{b_j}$ as endpoints. We denote by $\mathbb{P}$ the law of $(\eta^{(1)},\ldots,\eta^{(N)})$ and by $\mathbb{E}$ the corresponding expectation.

Let $F$ be any bounded continuous function on the space $(X_N, \dist)$. We claim that 
\begin{align} \label{eqn::FKIsing_cvg_aux1}
	\mathbb{E}\left[F(\eta^{(1)},\ldots,\eta^{(N)})\right]= \sum_{\alpha\in \LP_N} \LM_{\unnested,\alpha}(2) \LZ_{\alpha}(\theta_1,\ldots,\theta_{2N}) \frac{\mathbb{E}_{\alpha}\left[\one\left\{\mathcal{W}(\bs{\eta};0)\right\}\CR(\mathbb{U}\setminus \bs{\eta};0)^{-1/8}F(\eta^{(1)},\ldots,\eta^{(N)})\right]}{\LG^{(1/8)}(\theta_1,\ldots,\theta_{2N})},
\end{align}
where the function $\LG^{(1/8)}$ is defined in Definition~\ref{def::partition_random_cluster}. 
Combining the claim~\eqref{eqn::FKIsing_cvg_aux1} with Lemma~\ref{lem::LZalpharwired}, we conclude that the law of $\eta^{(1)}$ is the same as radial $\SLE_{\kappa}$ in $(\U; \exp(\ii\theta_1); 0)$ weighted by the following local martingale, up to the first time $\exp(\ii\theta_{2})$ or $\exp(\ii\theta_{2N})$ is disconnected from the origin:
\begin{equation*}
	M_t(\LG^{(1/8)})=g_t'(0)^{1/8-\tilde{h}}\times\prod_{j=2}^{2N}\phi_t'(\theta_j)^h\times\LG^{(1/8)}(\xi_t, \phi_t(\theta_2), \ldots, \phi_t(\theta_{2N})).
\end{equation*}                                                                                             
This gives~\eqref{eqn::driving_function_FK_conditional} for $j=1$. For $j\in \{2,\ldots,2N\}$, the proof is essentially the same.

We now prove the claim~\eqref{eqn::FKIsing_cvg_aux1}. Let $\vartheta^{\delta}=\{\{a_1^{\delta},b_1^{\delta}\},\ldots,\{a_N^{\delta},b_N^{\delta}\}\}\in \LP_N$ be the link pattern given by $(\eta_1^{\delta},\ldots,\eta_{2N}^{\delta})$. For $1\leq j\leq N$, let $\eta^{(j,\delta)}$ be the curve in $(\eta_1^{\delta},\ldots,\eta_{2N}^{\delta})$ having $x_{a_j^{\delta}}^{\delta,\diamond}$ and $x_{b_j^{\delta}}^{\delta,\diamond}$ as endpoints. We denote by $\tilde{\mathbb{E}}^{\delta}$ the expectation with respect to the law of $(\eta^{(1,\delta)},\ldots,\eta^{(N,\delta)})$. Note that the law of $(\eta^{(1,\delta_n)},\ldots,\eta^{(N,\delta_n)})$ converges weakly to $\mathbb{P}$ in the metric~\eqref{eqn::metric_multiple} as $n\to \infty$, which implies 
\begin{align}\label{eqn::FKIsing_cvg_aux2}                                               
	\lim_{n\to \infty}\tilde{\mathbb{E}}^{\delta_n} \left[F(\eta^{(1,\delta_n)},\ldots,\eta^{(N,\delta_n)})\right]=\mathbb{E}\left[F(\eta^{(1)},\ldots,\eta^{(N)})\right].  
\end{align}
It follows from the domain Markv property of the FK-Ising model that 
\begin{align}\label{eqn::FKIsing_cvg_aux3}
\begin{split}
		&\tilde{\mathbb{E}}^{\delta} \left[F(\eta^{(1,\delta)},\ldots,\eta^{(N,\delta)})\right]\\
		&\qquad= \frac{\mathbb{E}_{\Omega}^{\delta}\left[\one \left\{\LW(\bs{\eta}^{\delta};z^{\delta})  \right\} \mathbb{P}_{\Omega_{\bs{\eta}^{\delta}}^{\delta}(z^{\delta}),\mathrm{w}}^{\delta}\left[z^{\delta}\leftrightarrow\partial \Omega_{\bs{\eta}^{\delta}}^{\delta}(z^{\delta})\right]F(\eta^{(1,\delta)},\ldots,\eta^{(N,\delta)})\right]}{\mathbb{P}_{\Omega}^{\delta}\left[z^{\delta}\leftrightarrow \cup_{j=1}^N(x_{2j-1}^{\delta}x_{2j}^{\delta})\right]}\\
		&\qquad=\frac{\mathbb{E}_{\Omega}^{\delta}\left[\one \left\{\LW(\bs{\eta}^{\delta};z^{\delta})  \right\} \mathbb{P}_{\Omega_{\bs{\eta}^{\delta}}^{\delta}(z^{\delta}),\mathrm{w}}^{\delta}\left[z^{\delta}\leftrightarrow\partial \Omega_{\bs{\eta}^{\delta}}^{\delta}(z^{\delta})\right]F(\eta^{(1,\delta)},\ldots,\eta^{(N,\delta)})\right]/\mathbb{P}^{\delta}_{\mathbb{U},\mathrm{w}}[0\leftrightarrow \partial \mathbb{U}^{\delta}]}{\mathbb{P}_{\Omega}^{\delta}\left[z^{\delta}\leftrightarrow \cup_{j=1}^N(x_{2j-1}^{\delta}x_{2j}^{\delta})\right]/\mathbb{P}^{\delta}_{\mathbb{U},\mathrm{w}}[0\leftrightarrow \partial \mathbb{U}^{\delta}]}.
\end{split}
\end{align}
The denominator was treated in Corollary~\ref{coro::limit_one_point}. For the numerator, one can proceed as in the proof of Corollary~\ref{coro::limit_one_point} to show that 
\begin{align}\label{eqn::FKIsing_cvg_aux4}
\begin{split}
&	\lim_{\delta\to 0}	\mathbb{E}_{\Omega}^{\delta}\left[\one \left\{\LW(\bs{\eta}^{\delta};z^{\delta})  \right\} \mathbb{P}_{\Omega_{\bs{\eta}^{\delta}}^{\delta}(z^{\delta}),\mathrm{w}}^{\delta}\left[z^{\delta}\leftrightarrow\partial \Omega_{\bs{\eta}^{\delta}}^{\delta}(z^{\delta})\right]F(\eta^{(1,\delta)},\ldots,\eta^{(N,\delta)})\right]/\mathbb{P}^{\delta}_{\mathbb{U},\mathrm{w}}[0\leftrightarrow \partial \mathbb{U}^{\delta}]\\
	&\qquad=  \sum_{\alpha\in \LP_N} \LM_{\unnested,\alpha}(2)\frac{\mathcal{Z}_{\alpha}(\theta_1,\ldots,\theta_{2N})}{\mathcal{F}(\theta_1,\ldots,\theta_{2N})}\mathbb{E}_{\alpha}\left[\one \left\{\LW(\bs{\eta};z)\right\}\CR(\Omega\setminus \bs{\eta};z)^{-1/8}F(\eta^{(1)},\ldots,\eta^{(N)}) \right].
\end{split}
\end{align}
Plugging~\eqref{eqn::coro::limit_one_point},~\eqref{eqn::FKIsing_cvg_aux3} and~\eqref{eqn::FKIsing_cvg_aux4} into~\eqref{eqn::FKIsing_cvg_aux2} gives the claim~\eqref{eqn::FKIsing_cvg_aux1}, as we set out to prove. 
\end{proof}

\appendix
\section{Proof of Lemma~\ref{lem::limit_one_point}}
\label{app::technical}

The goal of this appendix is to prove Lemma~\ref{lem::limit_one_point}. We first prove a coupling result in Lemma~\ref{lem::inward_coupling} and collect a result concerning the convergence of FK-Ising interface loops to $\mathrm{CLE}_{16/3}$ in Proposition~\ref{prop::CLE}. Then we complete the proof of Lemma~\ref{lem::limit_one_point}.

\begin{lemma} \label{lem::inward_coupling}
	Suppose that $z^{\delta}\in \delta\mathbb{Z}^2$, $\Omega^{\delta},\hat{\Omega}^{\delta}\subseteq \delta\mathbb{Z}^2$ satisfy $B_{10\epsilon}(z^{\delta})\subseteq \Omega^{\delta}\cap \hat{\Omega}^{\delta}$ for some $\epsilon>0$.  Fix $q\in [1,4)$ and let $a\in (\delta,\epsilon)$.	Then there exists a  constant $c_2\in (0,\infty)$ depending only on $q\in [1,4)$ such that the following holds. There exists a coupling $\PRCM^{\delta}$, between $\tilde{\Lambda}^{\delta}\sim \PRCM_{p_c(q),q,\Omega^{\delta}}^{\pi_1}[\cdot \cond \partial \Omega^{\delta}\leftrightarrow \partial B_{a}(z^{\delta})]$ and $\hat{\Lambda}^{\delta}\sim \PRCM_{p_c(q),q,\hat{\Omega}^{\delta}}^{\pi_2}[\cdot \cond \partial \hat{\Omega}^{\delta}\leftrightarrow \partial B_{a}(z^{\delta})]$, and an event $\LS$, such that, 
	\begin{itemize}
		\item first, $\LS$ is the event that there exists a common open circuit surrounding $z^{\delta}$ inside $A_{a,\epsilon}(z^{\delta})$ in both $\tilde{\Lambda}^{\delta}$ and $\hat{\Lambda}^{\delta}$, and we denote by $\gamma^{\delta}$ the outermost such open circuit;
		\item  second, 
		\begin{align*}
			\PRCM^{\delta}[\LS^c]\lesssim&\left(\frac{a}{\epsilon}\right)^{c_2};
		\end{align*}
		\item and third, if $\LS$ happens, then the status of edges inside of the region surrounded by $\gamma^{\delta}$ are the same under both configurations $\tilde{\Lambda}^{\delta}$ and $\hat{\Lambda}^{\delta}$.
	\end{itemize}
	As a consequence, we have
	\begin{equation} \label{eqn::asy_identi_arm_event}
		\left\vert\frac{\PRCM_{p_c(q),q,\Omega^{\delta}}^{\pi_1}[z^{\delta}\leftrightarrow \partial \Omega^{\delta} \cond \partial \Omega^{\delta}\leftrightarrow \partial B_{a}(z^{\delta})]}{\PRCM_{p_c(q),q,\hat{\Omega}^{\delta}}^{\pi_2}[z^{\delta}\leftrightarrow \partial \hat{\Omega}^{\delta} \cond \partial \hat{\Omega}^{\delta}\leftrightarrow \partial B_{a}(z^{\delta})]}  -1\right\vert \lesssim  \left(\frac{a}{\epsilon}\right)^{c_2} .
	\end{equation}
\end{lemma}
\begin{proof}
	One can adopt the strategy in~\cite[Proof of Lemma~2.1]{CamiaPercolationCFT} to show the existence of the coupling $\PRCM^{\delta}$, with the FKG inequality and RSW estimates for critical site percolation replaced by these two properties for the critical random cluster model with cluster weight $q\in[1,4)$, and with the exploration starting from $z^{\delta}$ replaced by the exploration starting from $\partial B_{\epsilon}(z^{\delta})\cap \delta\mathbb{Z}^2$. 
	
	We then show the estimate~\eqref{eqn::asy_identi_arm_event}.  Note that
	\begin{align*}
		\PRCM_{p_c(q),q,\Omega^{\delta}}^{\pi_1}[z^{\delta}\leftrightarrow \partial \Omega^{\delta} \cond \partial \Omega^{\delta}\leftrightarrow \partial B_{a}(z^{\delta}),\; \LS]=&\PRCM_{p_c(q),q,\Omega^{\delta}}^{\pi_1}[z^{\delta}\leftrightarrow \gamma^{\delta}\cond  \partial \Omega^{\delta}\leftrightarrow \partial B_{a}(z^{\delta}),\; \LS],\\
		\PRCM_{p_c(q),q,\hat{\Omega}^{\delta}}^{\pi_2}[z^{\delta}\leftrightarrow \partial \hat{\Omega}^{\delta} \cond  \partial \Omega^{\delta}\leftrightarrow \partial B_{a}(z^{\delta}),\; \LS]=&\PRCM_{p_c(q),q,\hat{\Omega}^{\delta}}^{\pi_2}[z^{\delta}\leftrightarrow \gamma^{\delta}\cond  \partial \Omega^{\delta}\leftrightarrow \partial B_{a}(z^{\delta}),\; \LS].
	\end{align*}
	Let $B^{\delta}_{a}(z^{\delta}):=B_{a}(z^{\delta})\cap\delta\mathbb{Z}^2$. Consequently, thanks to the existence of the coupling $\PRCM^{\delta}$, we have 
	\begin{align}
		\begin{split}\label{eqn::asy_identi_arm_event_aux}
			&\left\vert	\PRCM_{p_c(q),q,\Omega^{\delta}}^{\pi_1}[z^{\delta}\leftrightarrow \partial \Omega^{\delta} \cond \partial \Omega^{\delta}\leftrightarrow \partial B_{a}(z^{\delta})] -\PRCM_{p_c(q),q,\hat{\Omega}^{\delta}}^{\pi_2}[z^{\delta}\leftrightarrow \partial \hat{\Omega}^{\delta} \cond \partial \hat{\Omega}^{\delta}\leftrightarrow \partial B_{a}(z^{\delta})]\right\vert \\
			&\qquad\qquad\qquad\qquad\qquad\leq \PRCM^{\delta}[\mathcal{S}^c]\times \mu^{1}_{p_c(q),q,B_{a}^{\delta}(z^{\delta})}\left[z^{\delta}\leftrightarrow \partial B_{a}^{\delta}(z^{\delta})\right]\\
			&\qquad\qquad\qquad\qquad\qquad\lesssim \left(\frac{\epsilon}{a}\right)^{c_2} \times \mu^{1}_{p_c(q),q,B_{a}^{\delta}(z^{\delta})}\left[z^{\delta}\leftrightarrow \partial B_{a}^{\delta}(z^{\delta})\right],
		\end{split}
	\end{align}
	where $\PRCM^{1}_{p_c(q),q,B_{a}^{\delta}(z^{\delta})}$ is the critical random cluster model on $B_{a}^{\delta}(z^{\delta})$ with the wired boundary condition. 
	Then~\eqref{eqn::asy_identi_arm_event} follows from~\eqref{eqn::asy_identi_arm_event_aux} and a standard application of the FKG inequality and RSW estimates. 
\end{proof}

\begin{proposition} \label{prop::CLE}\textnormal{(\cite[Theorem~1.1]{KemppainenSmirnovFullLimitFKIsing}, ~\cite[Theorem~1.1]{KemppainenSmirnovBoundaryTouchingLoopsFKIsing})}
	Fix a bounded simply connected domain $\Omega$  and suppose that a sequence of admissible medial domains $\Omega^{\delta,\diamond}$ converges to $\Omega$ in the Carath\'{e}odory sense. Consider the critical FK-Ising model on the primal domain $\Omega^{\delta}$ with the wired boundary condition. Let $\Gamma^{\delta}$ be the collection of loops in the loop representation. Then the law of $\Gamma^{\delta}$ converges weakly under the topology induced by $\Dist$ in~\eqref{eqn::LE_metric} to the law of the nested $\mathrm{CLE}_{16/3}$ on $\Omega$; we denote by $\mathbb{P}_{\Omega,\mathrm{w}}$ the latter law. 
\end{proposition}

\begin{proof}[Proof of Lemma~\ref{lem::limit_one_point}]
	Without loss of generality, we may assume that $z^{\delta}=z=0$. 
	
Choose a decreasing sequence $\{a_k\}_{k=1}^{\infty}$ such that $\lim_{k\to \infty}a_k=0$. For large enough $k$, write
\begin{align}\label{limit_one_point_aux1}
\frac{\mathbb{P}^{\delta}_{\Omega,\mathrm{w}}[0\leftrightarrow \partial \Omega^{\delta}]}{\mathbb{P}^{\delta}_{\mathbb{U},\mathrm{w}}[0\leftrightarrow \partial \mathbb{U}^{\delta}]}= \underbrace{\frac{\mathbb{P}^{\delta}_{\Omega,\mathrm{w}}[\partial B_{a_k}(0)\leftrightarrow \partial \Omega^{\delta}]}{\mathbb{P}^{\delta}_{\mathbb{U},\mathrm{w}}[\partial B_{a_k}(0)\leftrightarrow \partial \mathbb{U}^{\delta}]}}_{T_1^{\delta}}\times \underbrace{\frac{\mathbb{P}^{\delta}_{\Omega,\mathrm{w}}[0\leftrightarrow \partial \Omega^{\delta}\vert\partial B_{a_k}(0)\leftrightarrow \partial \Omega^{\delta}]}{\mathbb{P}^{\delta}_{\mathbb{U},\mathrm{w}}[0\leftrightarrow \partial \mathbb{U}^{\delta}\vert \partial B_{a_k}(0)\leftrightarrow \partial \mathbb{U}^{\delta}]}}_{T_2^{\delta}}.
\end{align}
As explained in~\cite[Section~2.1]{CamiaPercolationCFT}, in the discrete, the events $\{\partial B_{a_k}(0)\leftrightarrow \partial \Omega^{\delta}\}$ and $\{\partial B_{a_k}(0)\leftrightarrow \partial \mathbb{U}^{\delta}\}$ can be expressed in terms of interface loops in the loop representation; moreover, their analog events in the scaling limit defined using the $\mathrm{CLE}$ loops, which we denote by $\{\partial B_{a_k}\leftrightarrow \partial \Omega\}$ and $\{\partial \Omega_{a_k}\leftrightarrow \partial \mathbb{U}\}$, respectively, are continuity events. It then follows from Proposition~\ref{prop::CLE} that 
\begin{align}\label{limit_one_point_aux2}
	\lim_{\delta\to 0} T_1^{\delta} = \frac{\mathbb{P}_{\Omega,\mathrm{w}}[\partial B_{a_k}(0)\leftrightarrow \partial \Omega]}{\mathbb{P}_{\mathbb{U},\mathrm{w}}[\partial B_{a_k}(0)\leftrightarrow \partial \mathbb{U}]}. 
\end{align} 
For the term $T_2^{\delta}$, Lemma~\ref{lem::inward_coupling} implies that there exist two constants $c_2,c_3\in (0,\infty)$ such that 
\begin{align} \label{limit_one_point_aux3}
1-c_3\left(\frac{a_k}{\epsilon}\right)^{c_2}\leq \liminf_{\delta\to 0}T_2^{\delta}\leq \limsup_{\delta\to0 }T_2^{\delta}\leq 1+c_3\left(\frac{a_k}{\epsilon}\right)^{c_2}.  
\end{align}
A standard application of the FKG inequality and RSW estimates implies that there are two constants $c_4,c_5\in (0,\infty)$ such that 
\begin{align*}
	c_4\leq  \frac{\mathbb{P}_{\Omega,\mathrm{w}}[\partial B_{a_k}(0)\leftrightarrow \partial \Omega]}{\mathbb{P}_{\mathbb{U},\mathrm{w}}[\partial B_{a_k}(0)\leftrightarrow \partial \mathbb{U}]}\leq c_5,\quad \text{for all }k\geq 1. 
\end{align*}
So we can choose some subsequence $\{a_{k_j}\}_{j=1}^{\infty}$ such that 
\begin{align}\label{limit_one_point_aux4}
	\lim_{j\to \infty}  \frac{\mathbb{P}_{\Omega,\mathrm{w}}[\partial B_{a_{k_j}}(0)\leftrightarrow \partial \Omega]}{\mathbb{P}_{\mathbb{U},\mathrm{w}}[\partial B_{a_{k_j}}(0)\leftrightarrow \partial \mathbb{U}]}=V, \quad \text{for some }V\in (0,\infty). 
\end{align}
Plugging~\eqref{limit_one_point_aux2},~\eqref{limit_one_point_aux3} and~\eqref{limit_one_point_aux4} into~\eqref{limit_one_point_aux1} gives 
\begin{align*}
	\phi(\Omega;0):=\lim_{\delta\to 0 }\frac{\mathbb{P}^{\delta}_{\Omega,\mathrm{w}}[0\leftrightarrow \partial \Omega^{\delta}]}{\mathbb{P}^{\delta}_{\mathbb{U},\mathrm{w}}[0\leftrightarrow \partial \mathbb{U}^{\delta}]}=V\in (0,\infty),
\end{align*}
which also implies that the value $V$ is independent of the choice of the subsequence $\{a_{k_j}\}_{j=1}^{\infty}$ so that we actually have 
\begin{align} \label{lim_one_point_aux5}
\lim_{k\to \infty} \frac{\mathbb{P}_{\Omega,\mathrm{w}}[\partial B_{a_k}(0)\leftrightarrow \partial \Omega]}{\mathbb{P}_{\mathbb{U},\mathrm{w}}[\partial B_{a_k}(0)\leftrightarrow \partial \mathbb{U}]}=\phi(\Omega;0). 
\end{align}

Now we show that 
\begin{align} \label{lim_one_point_aux6}
	\phi(\Omega;0)=\CR(\Omega;0)^{-1/8}. 
\end{align}
First, assume that $\Omega= R \mathbb{U}$ for some $R>0$. Without loss of generality, we may assume that $R>1$. 
On the one hand, according to~\cite[Proof of Theorem~2]{SchrammSheffieldWilsonConformalRadii} (the first displayed equation in the proof), 
\begin{align*}
	\mathbb{P}_{\mathbb{U},\mathrm{w}} \left[\partial B_{r}(0)\leftrightarrow \partial \mathbb{U}\right] \asymp r^{1/8},\quad \text{as }r\to 0. 
\end{align*}
On the other hand, according to~\eqref{lim_one_point_aux5} and the scale invariance of $\mathrm{CLE}_{16/3}$, we have 
\begin{align} \label{lim_one_point_aux7}
\lim_{r\to 0} \frac{\mathbb{P}_{\mathbb{U},\mathrm{w}}[\partial B_{r/R}(0)\leftrightarrow \partial \mathbb{U}]}{\mathbb{P}_{\mathbb{U},\mathrm{w}}[\partial B_{r}(0)\leftrightarrow \partial\mathbb{U}]}	=\lim_{r\to 0} \frac{\mathbb{P}_{R\mathbb{U},\mathrm{w}}[\partial B_r(0)\leftrightarrow R\partial \mathbb{U}]}{\mathbb{P}_{\mathbb{U},\mathrm{w}}[\partial B_{r}(0)\leftrightarrow \partial\mathbb{U}]}=\phi(R\mathbb{U};0).
\end{align}
For $k> 1$, we write
\begin{align*}
	\mathbb{P}_{\mathbb{U},\mathrm{w}}[\partial B_{1/R^k}(0) \leftrightarrow \partial \mathbb{U}] =\frac{\mathbb{P}_{\mathbb{U},\mathrm{w}}[\partial B_{1/R^k}(0) \leftrightarrow \partial \mathbb{U}] }{\mathbb{P}_{\mathbb{U},\mathrm{w}}[\partial B_{1/R^{k-1}}(0) \leftrightarrow \partial \mathbb{U}] }\frac{\mathbb{P}_{\mathbb{U},\mathrm{w}}[\partial B_{1/R^{k-1}}(0) \leftrightarrow \partial \mathbb{U}] }{\mathbb{P}_{\mathbb{U},\mathrm{w}}[\partial B_{1/R^{k-2}}(0) \leftrightarrow \partial \mathbb{U}] }\cdots \frac{\mathbb{P}_{\mathbb{U},\mathrm{w}}[\partial B_{1/R}(0) \leftrightarrow \partial \mathbb{U}] }{\mathbb{P}_{\mathbb{U},\mathrm{w}}[\partial B_{1}(0) \leftrightarrow \partial \mathbb{U}] },
\end{align*}
where 
\begin{align*}
	\mathbb{P}_{\mathbb{U},\mathrm{w}}[\partial B_{1}(0) \leftrightarrow \partial \mathbb{U}] :=1. 
\end{align*}
Using~\eqref{lim_one_point_aux7} and the convergence of the Ces\`aro mean gives 
\begin{align}\label{lim_one_point_aux8}
\lim_{k\to \infty} \frac{1}{k} \log \mathbb{P}_{\mathbb{U},\mathrm{w}}[\partial B_{1/R^k}(0)\leftrightarrow 0]= \log \phi(R\mathbb{U};0). 
\end{align}
Comparing~\eqref{lim_one_point_aux6} with~\eqref{lim_one_point_aux8} gives 
\begin{align} \label{lim_one_point_9}
	\phi(R\mathbb{U};0)=R^{-1/8}.
\end{align}

Next, we consider the general case. Let $\varphi$ be the conformal map from $\mathbb{U}$ onto $\Omega$ such that $\varphi(0)=0$ and $\varphi'(0)>0$, then we have $\varphi'(0)=\CR(\Omega; 0)$. 
Let $A_{r_k,R_k}(0)$ be the thinnest annulus whose closure contains the symmetric difference\footnote{If $B_{a_k/\varphi'(0)}(0)=\varphi^{-1}(B_{a_k}(0))$, then define $r_k=R_k=a_k/\varphi'(0)$. } between $B_{a_k/\varphi'(0)}(0)$ and $\varphi^{-1}(B_{a_k}(0))$. Note that 
\begin{align}\label{eqn::derivarive}
	\lim_{k\to \infty} \frac{r_k}{a_k}=\lim_{k\to \infty} \frac{R_k}{a_k}= 1/\varphi'(0). 
\end{align}
It follows from~\eqref{lim_one_point_aux5} and the conformal invariance of $\mathrm{CLE}_{16/3}$ that 
\begin{align}\label{lim_one_point_10}
\begin{split}
		\phi(\Omega,0)=&\lim_{k\to \infty} \frac{\mathbb{P}_{\Omega,\mathrm{w}}[\partial B_{a_k}(0)\leftrightarrow \partial \Omega]}{\mathbb{P}_{\mathbb{U},\mathrm{w}}[\partial B_{a_k}(0)\leftrightarrow \partial \mathbb{U}]}=\lim_{k\to \infty} \frac{\mathbb{P}_{\mathbb{U},\mathrm{w}}[\partial \varphi^{-1}(B_{a_k}(0))\leftrightarrow \partial \mathbb{U}]}{\mathbb{P}_{\mathbb{U},\mathrm{w}}[\partial B_{a_k}(0)\leftrightarrow \partial\mathbb{U}]}\\
	\leq & \lim_{k\to \infty} \frac{\mathbb{P}_{\mathbb{U},\mathrm{w}}[\partial B_{R_k}(0)\leftrightarrow \partial \mathbb
		U]}{\mathbb{P}_{\mathbb{U},\mathrm{w}}[\partial B_{a_k}(0)\leftrightarrow \partial\mathbb{U}]}=\underbrace{\lim_{k\to \infty} \frac{\mathbb{P}_{\varphi'(0)\mathbb{U},\mathrm{w}}[\partial B_{a_k}(0)\leftrightarrow \varphi'(0)\partial\mathbb{U}]}{\mathbb{P}_{\mathbb{U},\mathrm{w}}[\partial B_{a_k}(0)\leftrightarrow\partial \mathbb{U}]}}_{T_3^{(k)}}\times \underbrace{\frac{\mathbb{P}_{\varphi'(0)\mathbb{U},\mathrm{w}}\left[\varphi'(0)R_k\partial \mathbb{U}\leftrightarrow \varphi'(0)\partial\mathbb{U}\right]}{\mathbb{P}_{\varphi'(0)\mathbb{U},\mathrm{w}}[\partial B_{a_k}(0)\leftrightarrow \varphi'(0)\partial \mathbb
			U]}}_{T_4^{(k)}}. 
\end{split}
\end{align}
According to~\eqref{lim_one_point_aux5} and~\eqref{lim_one_point_9}, we have
\begin{align}\label{lim_one_point_10_aux1}
	\lim_{k\to \infty}T_3^{(k)}=\phi(\varphi'(0)\mathbb{U};0)=\varphi'(0)^{-1/8}=\CR(\Omega;0)^{-1/8}. 
\end{align}
According to the observation~\eqref{eqn::derivarive} and the fact that the boundary three-arm exponent for the FK-Ising model equals $2$ (see e.g.,~\cite[Theorem~2]{WuPolychromaticArmFKIsing}), which is strictly bigger than $1$, we have
\begin{align}\label{lim_one_point_10_aux2}
	\lim_{k\to \infty}T_4^{(k)}=1.
\end{align}
Plugging~\eqref{lim_one_point_10_aux1} and~\eqref{lim_one_point_10_aux2} into~\eqref{lim_one_point_10} gives
\begin{align}\label{lim_one_point_11}
	\phi(\Omega,0)\leq \CR(\Omega;0)^{-1/8}.
\end{align}
Similarly, one can show that 
\begin{align}\label{lim_one_point_12}
	\phi(\Omega,0)\geq \CR(\Omega;0)^{-1/8}. 
\end{align}
Combining~\eqref{lim_one_point_11} with~\eqref{lim_one_point_12} gives~\eqref{lim_one_point_aux6} and completes the proof. 
\end{proof}

\newcommand{\etalchar}[1]{$^{#1}$}

\end{document}